\documentclass[12pt]{article}
\usepackage{amsmath}
\usepackage{amssymb}
\usepackage{amsthm}
\usepackage{amscd}
\usepackage{amsxtra}
\usepackage{verbatim}
\usepackage{xcolor}
\usepackage{color}
\usepackage{enumerate}
\usepackage{mathrsfs}
\usepackage[all]{xy} 

\usepackage{array,arydshln}
\usepackage{bm}

\allowdisplaybreaks

\numberwithin{equation}{section}

\definecolor{dblue}{rgb}{0,0,0.45}
\definecolor{red}{rgb}{0.7,0,0}

 \RequirePackage{geometry}
\newtheorem{theorem}{Theorem}[section]
\newtheorem{lemma}[theorem]{Lemma}
\newtheorem*{lemma*}{Lemma}

\newtheorem{proposition}[theorem]{Proposition}

\theoremstyle{definition}

\newtheorem{remark}[theorem]{Remark}

\theoremstyle{remark}

\newcommand{\N}{{\mathbb N}}
\newcommand{\R}{{\mathbb R}}

\newcommand{\Q}{{\mathbb Q}}

\newcommand{\cD}{{\mathcal D}}
\newcommand{\cE}{{\mathcal E}}

\newcommand{\cH}{{\mathcal H}}

\newcommand{\cK}{{\mathcal K}}
\newcommand{\cL}{{\mathcal L}}
\newcommand{\cM}{{\mathcal M}}
\newcommand{\cN}{{\mathcal N}}
\newcommand{\cP}{{\mathcal P}}

\newcommand{\cS}{{\mathcal S}}

\newcommand{\cV}{{\mathcal V}}
\newcommand{\cW}{{\mathcal W}}

\newcommand{\cZ}{{\mathcal Z}}

\newcommand{\al}{\alpha}

\newcommand{\la}{\langle}
\newcommand{\ra}{\rangle}
\newcommand{\nn}{\nonumber}
\newcommand{\ve}{\varepsilon}

\newcommand{\pfbM}{O (\cM; \cD \oplus \cD^{\perp})}

\date{}
\begin{document}

\title{Large deviations for small noise 
hypoelliptic diffusion bridges
on sub-Riemannian manifolds
}
\author{   Yuzuru \textsc{Inahama} 
}
\maketitle

\begin{abstract}
\noindent
In this paper we study a large deviation principle
of Freidlin-Wentzell type for pinned hypoelliptic diffusion measures 
associated with a natural 
sub-Laplacian on a compact sub-Riemannian manifold.
To prove this large deviation principle,
 we use rough path theory and manifold-valued
Malliavin calculus.
\end{abstract}

\section{Introduction}

In the theory of  stochastic differential equations (SDEs),
small noise problems SDEs
are considered very important and have been studied intensively and extensively.
A large deviation principle (LDP) associated with them is 
called Freidlin-Wentzell's LDP.
One of its typical formulations is as follows.
Let $\cM$ be a Euclidean space or a manifold 
and let $V_i$, $0 \le i \le d$, be sufficiently nice vector fields on $\cM$. 
For a standard $d$-dimensional Brownian motion $(w_t)_{0\le t\le 1}$, 
consider the following Stratonovich-type SDE:
\[
dX_t^{\ve} =\ve \sum_{i=1}^d  V_i (X_t^{\ve})\circ dw_t^i   +\ve^2 V_0 (X_t^{\ve}) dt,
\qquad
X_0^{\ve} =x.
\]
Here, $x  \in\cM$ is a given initial point and $0<\ve \ll 1$ is a small parameter.
Note that $X^{\ve} =(X_t^{\ve})_{0\le t \le 1}$ is the diffusion process 
associated with the generator $\ve^2 (\tfrac12 \sum_{i=1}^d  V_i^2 +V_0)$
and the starting point $x$.
Then, as is well-known, the law of $X^{\ve}$ satisfies an LDP as $\ve\searrow 0$.

Let us consider the case that $X^{\ve}_t$ has a (sufficiently nice)
strictly positive density with respect to a reference measure on $\cM$ 
(e.g. the Lebesgue measure when $\cM$ is a Euclidean space).
Then, the pinned diffusion process from $x$ to $a$
associated with the above generator exists, where $a$ 
is a given end point.
It seems quite natural to ask whether an LDP of Freidlin-Wentzell type
holds for these scaled pinned 
diffusion measures  as $\ve\searrow 0$.
In this work we take up this problem.

Although there are a large number of papers on the standard version 
of Freidlin-Wentzell's LDP, 
only a few paper have been published on this type of LDP.
To the author's knowledge, the first one is Hsu \cite{hsu}.
He proved a pinned version of Freidlin-Wentzell's LDP 
for pinned Brownian motion on a compact Riemannian manifold
(i.e., the pinned diffusion process associated with 
one half of Laplace-Beltrami operator).
Then, by studying SDEs under a suitable bracket-generating condition 
on the coefficient vector fields,
the author  \cite{in1, in2}  proved this type of LDP 
in the Euclidean setting for rather general pinned diffusion processes.
His method is a combination of rough path theory 
and quasi-sure analysis, which is a potential theoretic part of Malliavin calculus.
Also, Bailleul \cite{bai} proved this type of LDP on a compact manifold
for pinned diffusions
associated with the sum-of-square type generator 
as above with a  suitable bracket-generating condition 
on these vector fields.
He combined a probabilistic method (rough path theory) and 
and an analytic method (Sanchez-Calle's estimate for 
the semigroup generated by 
a sum-of-square type operator).

The purpose of this paper is to prove an analogous LDP for 
the pinned diffusion process on  a sub-Riemannian manifold $\cM$
associated with the generator $\ve^2  (\Delta_{\mathrm{sub}}/2 +V)$, 
where $\Delta_{\mathrm{sub}}$ is a natural ``div-grad type"
sub-Laplacian on $\cM$
and $V$ is an arbitrary smooth vector field on $\cM$.
Our proof is basically similar to those in  the preceding works \cite{in1, in2}.
However,
 there are two new ingredients in this work.
First, in order to realize the $\ve^2  (\Delta_{\mathrm{sub}}/2 +V)$-diffusion process via an SDE, we use Eells-Elworthy's construction on
a frame bundle over $\cM$.
Second, since we work on the manifold $\cM$ and its frame bundle,
we need manifold-valued Malliavin calculus developed in Taniguchi \cite{ta}.  
 
The organization of this paper is as follows.
In Section \ref{sec.result} we formulate our LDP precisely
and then state our main result (Theorem \ref{thm.mainQ}).
Section \ref{sec.EE} is devoted to recalling the
stochastic parallel transport over a sub-Riemannian manifold.
Stochastic tools such as Malliavin calculus, quasi-sure analysis,
and rough path theory are collected in  Section \ref{sec.prob}.
In Section \ref{sec.LDP_wd} we provide an LDP 
for the rough path lifts of certain positive Watanabe distributions
on the geometric rough path space (Theorem \ref{tm.ldp.dlt}).
Our main result is almost immediate from this, thanks to Lyons' 
continuity theorem. 
The lower estimate of the LDP in Theorem \ref{tm.ldp.dlt}
is proved in Section \ref{sec.lower}, 
while the upper estimate is proved in Section \ref{sec.upper}.
In Appendix \ref{sec.appen} we show the strict positivity 
of the heat kernel associated with $\ve^2  (\Delta_{\mathrm{sub}}/2 +V)$, which ensures the well-definedness of the pinned diffusion measure.

Throughout this paper we will use the following notation.
\begin{itemize}
\item
$\N_+ =\{ 1,2,3, \ldots\}$ and $\N =\{ 0\} \cup \N_+$.
The set of all real numbers is denoted by $\R$.
\item
The time interval of (rough) paths
 is basically $[0,1]$ unless otherwise specified.
\item
Let $U$ be an open set of a manifold $\cN$
and $\cV$ be a vector (or a fiber) bundle over $U$.
We denote by $\Gamma (U; \cV)$ the set of all smooth sections 
of $\cV$ on $U$. 
When $U = \cN$, we will often simply write $\Gamma (\cV)$.
For instance, we will write $\Gamma (T\cN)$  or 
$\Gamma (\cN;T\cN)$ for the set of all smooth vector fields on $\cN$. 
   \item
      For a manifold $\cN$, $y\in \cN$, 
            and a subbundle $\cV$ of $T\cN$ with a metric, we set  
\begin{align*}
\cH_y (\cN, \cV) 
&=\{  \gamma \colon [0,1] \to \cN
\mid \mbox{absolutely continuous,} \quad  \gamma_0 =y,
\\
&\qquad\qquad  
  \mbox{$ \gamma^\prime_t \in \cV_{ \gamma_t}$ for almost all $t$}, \quad
\cE ( \gamma) :=\int_0^1 | \gamma^\prime_t|_{\cV}^2 dt <\infty
\}.
\end{align*}
We call $\cE (\gamma)$ the energy of $ \gamma$.
 \item 
Cameron-Martin space $\cH^d$ over $\R^d$ is
a real Hilbert space  defined by
\begin{align*}
\cH^d
&=\{ h \colon [0,1] \to \R^d
\mid \mbox{absolutely continuous,} \quad  h_0 =0,
\\
&\qquad\qquad  
  \mbox{with} \quad
\| h\|_{\cH^d}^2:=\int_0^1 | h^\prime_t|_{\R^d}^2 dt <\infty
\}.
\end{align*}
In other words, $\cH^d = \cH_0 (\R^d, T\R^d)$
and $\cE (h)=\| h\|_{\cH^d}^2$.
When $d$ is obvious from the context, we simply write $\cH$.

\item
The classical $d$-dimensional Wiener space 
$(\cW, \cH, \mu)=(\cW^d, \cH^d, \mu^d)$ is defined as follows.
(1)~$\cH$ is Cameron-Martin Hilbert space as above. 
(2)~$\cW$ is the Banach space of all continuous paths 
from $[0,1]$ to $\R^d$ which start at the origin.
The topology of $\cW$ is that of uniform convergence as usual.
(3)~$\mu$ is the $d$-dimensional Wiener measure.
A generic element of $\cW$ is denoted by $w$.
The coordinate process $(w_t)_{0\le t \le 1}$ defined on $(\cW, \mu)$
is called the canonical realization
of $d$-dimensional Brownian motion.

      \item
       Let $Z_1, \ldots, Z_m$ be smooth vector fields on an open subset $U$ of a manifold $\cN$. We set $\Sigma^1 = \{Z_1, \ldots, Z_m \}$
       and $\Sigma^k =\{  [Z_i, Y]  \mid Y\in \Sigma^{k-1}\}$
for $k\ge 2$, recursively.       
We next set for $k \ge 1$
 \[
 {\rm Lie}^{(k)} (Z_1, \ldots, Z_m) ={\rm span} \left[ 
  \cup_{j=1}^k \Sigma^j
    \right]
    \,\mbox{ and }\,
     {\rm Lie} (Z_1, \ldots, Z_m) ={\rm span} \left[ 
  \cup_{j=1}^\infty \Sigma^j
    \right].
 \]       
 Here, ${\rm span}~A$ means the linear span of $A$.
For $x \in U$, we  set
   \[
    {\rm Lie}^{(k)} (Z_1, \ldots, Z_m) (x) = \{ Y (x) \mid 
      Y \in {\rm Lie}^{(k)} (Z_1, \ldots, Z_m)\}
      \subset 
      T_x \cN
        \]
        and also set ${\rm Lie} (Z_1, \ldots, Z_m) (x)\subset 
      T_x \cN$ in the same way.
 \end{itemize}        

As for the last item, one should note the following simple fact. 
Suppose that
  there is a smooth function $F\colon U \to {\rm GL} (m, \R)$
such that
\[
[ Z_1, \ldots, Z_m] =[ \hat{Z}_1, \ldots, \hat{Z}_m] F
\qquad \mbox{on $U$.}
\]
Then, for every $x\in U$, we have
${\rm Lie}^{(k)} (Z_1, \ldots, Z_m) (x) = {\rm Lie}^{(k)} (\hat{Z}_1, \ldots, \hat{Z}_m) (x)$ for every  $k \ge 1$ and therefore
${\rm Lie} (Z_1, \ldots, Z_m) (x) = {\rm Lie} (\hat{Z}_1, \ldots, \hat{Z}_m) (x)$.

\section{Setting and main result}
\label{sec.result}

First we recall the essentials of sub-Riemannian 
geometry following \cite{ri}.
We say that 
 $(\cM,\cD,g)$ is a sub-Riemannian manifold if   
(i)~$\cM$ is a connected, smooth manifold of
dimension $n$, 
(ii)~$\cD\subset T\cM$, $T\cM$ being the tangent bundle of $\cM$,
is a smooth distribution of constant rank $d~(1 \le d \le n)$
which satisfies the H\"ormander condition at every $x \in \cM$
and
(iii)~$g=(g_x)_{x\in \cM}$,
where each $g_x$ is an inner product on the fiber $\cD_x$,
and $x\mapsto g_x$ is smooth.
(When there is no risk of confusion, we simply say that 
$\cM$ is a sub-Riemannian manifold.)
When $n=d$, this definition coincides with that of a 
(connected) Riemannian manifold.
Throughout this paper $\cM$ is assumed to be compact.

The precise statement of 
 the H\"ormander condition on $\cD$ at $x \in \cM$ is as follows:
If $\{Z_1, \ldots, Z_d\}$ is a local frame of $\cD$ 
over a coordinate neighborhood $U$ around $x$, 
then 
${\rm Lie} (Z_1, \ldots, Z_d) (x)  = T_{x} \cM$.
As is well-known, this condition does not depend on the choice of 
$U$ and $\{Z_1, \ldots, Z_d\}$.

Now we recall a ``div-grad type" sub-Laplacian 
on a  sub-Riemannian manifold $\cM$.
Let ${\bf vol}$ be a smooth volume on $\cM$, that is, 
${\bf vol}$ is a measure on $\cM$ whose restriction to every local coordinate chart 
is written as a strictly positive smooth density function 
times the Lebesgue measure on the chart.
We consider the second-order differential operator of the form
$\Delta_{\mathrm{sub}}= \mathrm{div}\circ \mathrm{grad}_{\cD}$,
where $\mathrm{grad}_{\cD}$ is the
 horizontal gradient in the direction of $\cD$
and $\mathrm{div}$ is the divergence with respect to ${\bf vol}$
(i.e. $\mathrm{div}= - (\mathrm{grad}_{\cD})^*$ at least formally, 
where the adjoint is taken with respect to ${\bf vol}$).

A continuous path $\gamma \colon [0,1] \to \cM$
is said to be an admissible path
\footnote{It is called a {\it horizontal} path in \cite{ri}.
We avoid this term, however, because the term ``{\it horizontal}" 
is also used in the theory of connections on a principal bundle.}
 if $\gamma$ is absolutely 
continuous, $ \gamma_t^{\prime} \in \cD_{\gamma_t}$ 
for almost all $t \in [0,1]$,
and 
\begin{equation}\label{def.energy}
\cE (\gamma) :=\int_0^1 
   | \gamma_t^{\prime}|^2_{g_{\gamma_t}} dt <\infty.
\end{equation}
We call $\cE (\gamma)$ the energy of $\gamma$.
(If $\gamma$ is not admissible, we set $\cE (\gamma)=+\infty$ by convention.)
By Chow-Rashevsky's theorem \cite[Theorem 1.14]{ri}, 
for every $x, y \in \cM$
there exists an admissible path $\gamma$ such that 
$\gamma_0 =x$ and $\gamma_1 =y$.

We define $d_{SR}\colon \cM \times \cM \to [0,\infty)$ by 
\begin{align*}
d_{SR} (x,y) &= \inf\{  \int_0^1  | \gamma_t^{\prime}|_{g_{\gamma_t}} dt
\mid \gamma \colon [0,1] \to \cM,
\mbox{ admissible with $\gamma_0 =x, \gamma_1 =y$}
\}.
\end{align*}
Then, $d_{SR} (x,y) <\infty$ for every $x, y \in \cM$. 
It is well-known that $d_{SR}$ becomes a distance on $\cM$,
which generates the same topology as the original 
manifold topology of $\cM$.
This is called the sub-Riemannian distance of $\cM$.
According to \cite[Proposition 2.1]{ri}, it holds that 
\begin{equation}\label{eq.ED^2}
d_{SR} (x,y)^2 
=
\inf\{  \cE (\gamma)  \mid \gamma \colon [0,1] \to \cM,
\mbox{ admissible with $\gamma_0 =x, \gamma_1 =y$}
\}.
\end{equation}


Let $V$ be any smooth vector field on $\cM$ (i.e. $V \in \Gamma (T\cM)$) and $\ve \in (0,1]$.
Then, the diffusion process on $\cM$ associated with 
$\ve^2 (\Delta_{\mathrm{sub}}/2 +V)$ starting at $x \in \cM$ 
has a density $p^{\ve}_t (x, a)$ with respect to ${\bf vol}(da)$
at every $t >0$.
In fact, $p^{\ve}_t (x, a) >0$ for all $x,a \in \cM$
and $t, \ve \in (0,1]$ and $a \mapsto p^{\ve}_t (x, a)$ is smooth
for all $x \in \cM$ and $t, \ve \in (0,1]$.

The pinned diffusion measure $\Q^{\ve}_{x,a}$ 
associated with $\ve^2 (\Delta_{\mathrm{sub}}/2 +V)$
from $x$ to $a$ is a unique probability measure on $C([0,1], \cM)$,
the continuous path space over $\cM$, such that
the following holds:
For every $k \ge 1$, $0=t_0 <t_1 <\cdots < t_k < t_{k+1}=1$
and $G \in C^\infty (\cM^k)$,
\begin{align*}
\lefteqn{
\int_{C([0,1], \cM)} 
 G (\xi_{t_1}, \ldots, \xi_{t_k}) \Q^{\ve}_{x,a} (d\xi)
}
\\
&= 
p_1^\ve (x,a)^{-1}
\int_{\cM^k} G (x_1, \ldots, x_k) \prod_{i=0}^k p^\ve_{t_{i+1} -t_i} (x_i, x_{i+1})
\prod_{i=1}^k {\rm vol} (dx_i)
\end{align*}
with the convention that $x_0 =x$ and $x_{k+1} =a$.
(We will see that $\Q^{\ve}_{x,a}$ does exist.)
Here, $C([0,1], \cM)$ is the set of all continuous paths on $\cM$
equipped with the compact-open topology.
The closed subset of all continuous paths
 which start at $x$ and end at $a$
is denoted by $C_{x,a}([0,1], \cM)$, in which $\Q^{\ve}_{x,a}$ 
is supported.

Now we state our main theorem.  
This can be viewed as a version of Freidlin-Wentzell type LDP
for pinned hypoelliptic pinned diffusion processes
on  a sub-Riemannian manifold.
As one can easily expect, the rate function equals one half of
the energy functional on the path space (up to an additive constant).
This theorem includes the main theorem of \cite{hsu} as a special case.
We will prove it in Section \ref{sec.LDP_wd} as a simple application of 
Theorem \ref{tm.ldp.dlt}.
The goodness of the rate function $J$ defined by 
\eqref{eq.0720-2} below is a part of our claim.
Note that $J$ actually attains its minimum $0$
because of its goodness and \eqref{eq.ED^2}.

\begin{theorem}\label{thm.mainQ}
Let $(\cM, \cD, g)$ be a compact sub-Riemannian manifold
with a smooth volume ${\bf vol}$.
For $x, a \in \cM$, $V\in \Gamma (T\cM)$ and $\ve \in (0,1]$,
let $\Q^{\ve}_{x,a}$ be the pinned diffusion measure as above. 
Then, $\{\Q^{\ve}_{x,a} \}_{0<\ve \le 1}$ satisfies an LDP 
on $C_{x,a}([0,1], \cM)$ as $\ve\searrow 0$ 
with the speed $\ve^2$
and the good rate function $J\colon C_{x,a}([0,1], \cM)\to [0,\infty]$,
where 
\begin{equation}\label{eq.0720-2}
J (\gamma)=
\frac12 \left\{
\cE (\gamma) - d_{SR} (x,a)^2
\right\}.
\end{equation}
Here, $\cE (\gamma)$ is the energy of $\gamma$  defined by \eqref{def.energy}.
\end{theorem}

\begin{remark}\label{rem.COtop}
Take any distance $d$ on $\cM$ which generates 
the topology of $\cM$. 
Then, 
\[
{\rm dist} (\gamma, \hat{\gamma}) 
:=\sup_{0\le t \le 1} d(\gamma_t, \hat{\gamma}_t)
\]
defines a distance on $C([0,1], \cM)$.
It is known that ${\rm dist}$ generates 
the compact-open topology of $C([0,1], \cM)$ regardless of the choice 
of $d$. 
Typical examples of $d$ include
(1) the sub-Riemannian distance $d_{SR}$ on $\cM$,
(2) the Riemannian distance on $\cM$ with respect to
 any Riemannian metric on $\cM$,
(3) the Euclidean distance of $\R^m$ restricted to $\cM$ for any 
embedding $\cM \hookrightarrow \R^m$.
\end{remark}

\begin{remark}\label{rem.strong}
In Theorem \ref{thm.mainQ} above and Proposition \ref{prop.subQ} 
below, the topology of the path spaces 
can be slightly strengthened as we now explain.
In this remark, $\alpha \in (1/3, 1/2)$.

Let $\cN$ be a compact smooth manifold and let
$\iota\colon \cN \hookrightarrow \R^m$ be 
an embedding  for some $m \in \N_+$.
We denote by $C^{\alpha -H} ([0,1], \R^m)$
 be the set of all $\alpha$-H\"older continuous paths  
 taking values in $\R^m$ and define 
 \[
 C^{\alpha -H} ([0,1], \cN) :=
  C^{\alpha -H} ([0,1], \R^m) \cap C([0,1], \cN),
 \]
whose distance is the restriction of the natural one 
of $C^{\alpha -H} ([0,1], \R^m)$.  

Let $\cN'$ be another compact smooth manifold and let
$\iota' \colon \cN' \hookrightarrow \R^{m'}$ be 
an embedding for some $m' \in \N_+$
and suppose that $\phi \colon \cN \to \cN'$ is a smooth map.
Noting that $\phi$ extends to a smooth map from $\R^{m}$ to $\R^{m'}$ 
with compact support, $\phi$ naturally induces a 
continuous map from $C^{\alpha -H} ([0,1], \cN)$ to 
$C^{\alpha -H} ([0,1], \cN')$ in an obvious way.
In particular, this topology of $C^{\alpha -H} ([0,1], \cN)$  is 
independent of the choice of $\iota$.

The reason why the LDP in Theorem \ref{thm.mainQ} (and in Proposition  \ref{prop.subQ} below) 
also holds on $C^{\alpha -H} ([0,1], \cM)$
is as follows.
When we derive Theorem \ref{thm.mainQ} from Theorem \ref{tm.ldp.dlt},
we consider a rough differential equation (RDE) on $\cP$, a principle bundle over $\cM$, and
embed $\cP$ into a Euclidean space
and then use Lyons' continuity theorem with respect to the $\alpha$-H\"older rough path topology, together with the contraction 
principle for LDPs.
So, one can easily see 
that our LDP holds on $C^{\alpha -H} ([0,1], \cM)$, too.
(However, we do not write this fact in our main theorem
since this set does not look very beautiful from the geometric viewpoint.)
\end{remark}

Before closing this section, we claim that
our method can reprove a very similar LDP in 
\cite{bai} on a compact manifold when the generator of the diffusion
process is of sum-of-squares type.

Let $\cM'$ be a compact smooth manifold
and $V_0, V_1, \ldots, V_k$ ($k \in \N_+$) be smooth vector fields
on $\cM'$.
Consider the second order differential operator
 $\ve^2 (V_0 + \tfrac12\sum_{i=1}^k V_i^2)$.
Assume the bracket generating condition that
 ${\rm Lie} (V_1, \ldots, V_k) (x) =T_x  \cM'$ for every $x \in \cM'$.
Then, the pinned diffusion measure from $x$ to $a$
associated with this operator 
exists uniquely, which is denoted by $\tilde\Q^{\ve}_{x,a}$.

The following proposition, together with Remark \ref{rem.strong},
 is the main result of \cite[Theorem 1]{bai}.
See \cite[Eq. (4)]{bai} for an explicit expression of $J'$,
which may not always have a deep geometric meaning.
\begin{proposition}\label{prop.subQ}
Let the notation and assumptions be as above.
Then, for every $x, a \in \cM'$,
$\{\tilde\Q^{\ve}_{x,a} \}_{0<\ve \le 1}$ satisfies an LDP 
on $ C_{x,a} ([0,1], \cM')$ as $\ve\searrow 0$
with the speed $\ve^2$ and
a good rate function $J'$.
\end{proposition}

\begin{proof}
We can prove this proposition using 
 the same method for Theorem \ref{thm.mainQ}.
 The actual proof of this proposition
 is much simpler than that of Theorem \ref{thm.mainQ}
 since we need not use a principal bundle over $\cM'$. 
 So, we omit the proof.
  \end{proof}

\begin{remark}
Suppose that
 $\Delta_{\mathrm{sub}}$ admits a sum-of-squares form
 in the following sense: there exist $k \ge d$ and 
$V_i \in \Gamma (T\cM)~(1\le i \le k)$
 such that the following condition holds:
\begin{equation}\label{eq.0720-1}
\Delta_{\mathrm{sub}} - \sum_{i=1}^k V_i^2 \,\, \in \Gamma (T\cM),
\quad
{\rm Lie} (V_1, \ldots, V_k) (x) =T_x  \cM' \mbox{ for every $x \in \cM'$}.
\end{equation}
When $\Delta_{\mathrm{sub}}$ admits the above expression,
it is in a sense true that the LDP in Proposition \ref{prop.subQ}
immediately implies the LDP in Theorem \ref{thm.mainQ}.
However, it does not seem easy to
obtain in this way
the explicit expression \eqref{eq.0720-2} of the rate function $J$
from $J'$. The reason is as follows.
In this case, $k$ equals the dimension of Brownian motion
that plays a key role in the proof
and  can be very large. Hence, a one-to-one correspondence of 
 Cameron-Martin paths and admissible paths on $\cM$ 
as in Proposition \ref{pr.iso_path} breaks down in general.
Therefore, even when \eqref{eq.0720-1} holds, we believe
 Theorem \ref{thm.mainQ}  is worth proving.
(Unfortunately, the author does not know for which sub-Riemannian
manifolds Condition \eqref{eq.0720-1} holds.
For instance, for the Laplace-Beltrami operator on a Riemannian
manifold, \eqref{eq.0720-1} is known to be satisfied.)
\end{remark}

\section{Stochastic parallel transport over sub-Riemannian manifold}
\label{sec.EE}

In this section we recall the stochastic parallel transport 
over a compact sub-Riemannian manifold $(\cM, \cD, g)$.
The history of Eells-Elworthy's construction of
It\^o's  stochastic parallel transport is quite long.
The Riemannian case is classical (see \cite[Chapter 2]{hsu_bk} or \cite[Section V-4]{iwbk} for example),
but non-Riemannian cases have also been studied.
For diffusion processes associated with 
semielliptic second-order differential operators, see \cite{ell, ellbk}.
For those associated with
sub-Laplacians on various kinds of sub-Riemannian manifolds,
see \cite{bfg, bhm, gt, kt, thal} among others.
In this section we will mainly follow \cite{gt}.


Now we construct a principle bundle over $\cM$ 
and a connection which are associated with the 
sub-Riemannian structure of $(\cM, \cD, g)$.
We write $n=\dim \cM$ and the rank of $d = {\rm rank}~\cD$ 
with $1 \le d \le n$.
Since the Riemannain case (i.e. the case $d=n$) is classical
and simpler, 
we only elaborate the case $1 \le d <n$ in this section.
But, the results in this section still hold for the case $d=n$.  

First, we take a subbundle $\cD^{\perp}=(\cD^{\perp}_x)_{x\in \cM}$ of the tangent bundle $T\cM$ such that 
\begin{equation}\label{eq.orth.dcmp}
T_x\cM =\cD_x\oplus\cD^{\perp}_x,
\qquad  x \in \cM
\end{equation}
holds. 
To make sure the existence of such $\cD^{\perp}$,
one just need to take any Riemannian metric of  $\cM$ and 
set $\cD^{\perp}_x$ to be the orthogonal complement of 
$\cD_x$ in $T_x\cM$ with respect to the metric.
The projections associated with \eqref{eq.orth.dcmp} 
are denoted by 
$\mathbf{pr}_x \colon T_x\cM \to \cD_x$
and 
$\mathbf{pr}^{\perp}_x \colon T_x\cM \to \cD^{\perp}_x$, respectively.
Then, $\mathbf{pr}= (\mathbf{pr}_x)_{x\in \cM}$ 
and $\mathbf{pr}^{\perp}= (\mathbf{pr}^{\perp}_x)_{x\in \cM}$
belong to $\Gamma (\mathbf{End} (T\cM, \cD))$ 
and $\Gamma (\mathbf{End} (T\cM, \cD^{\perp}))$, respectively.
For  any metric $h$ on $\cD^{\perp}$, we set
\[
\hat{g}_x \la v, \hat{v}\ra = g_x \la \mathbf{pr}_x v, \mathbf{pr}_x\hat{v}\ra 
+ h_x \la \mathbf{pr}^{\perp}_x v, \mathbf{pr}^{\perp}_x \hat{v}\ra,
\qquad
v, \hat{v} \in T_x\cM.
\]
Then, $\hat{g}$ is a Riemannian metric on $\cM$ which tames $g$,
that is, the restriction of $\hat{g}_x$ to $\cD_x \times \cD_x$ 
equals $g_x$. 
Moreover, the decomposition \eqref{eq.orth.dcmp}
is an orthogonal decomposition with respect to $\hat{g}_x$.

Next, choose metric Koszul connections $\tilde\nabla$  
and $\tilde\nabla^{\perp}$ on $\Gamma(\cD)$ and 
$\Gamma(\cD^{\perp})$ with respect to 
the metrics $g$ and $h$, respectively.
Note that they always exist. Define
\begin{equation}\label{def.nabla}
\nabla_X V = \tilde\nabla_X (\mathbf{pr} V)
+
\tilde\nabla^{\perp}_X (\mathbf{pr}^{\perp} V),
\quad
X, V \in \Gamma (T\cM).
\end{equation}
Since the projections are pointwise operations, 
 $\nabla$ is a $\hat{g}$-metric Koszul connection on $\Gamma(T\cM)$.
It is immediate from the definition that, 
for all $X \in \Gamma (T\cM)$, $Y \in \Gamma (\cD)$
and $Z \in \Gamma (\cD^{\perp})$, we have
$\nabla_X Y \in \Gamma (\cD)$ and  
$\nabla_X Z \in \Gamma (\cD^{\perp})$.

Now we introduce a principal bundle over $\cM$.
We choose any $(\cD^{\perp}, h)$ and  $\nabla$
(or $\tilde\nabla$ and $\tilde\nabla^{\perp}$)
as above and will fix them in what follows.
The product $O (d) \times O(n-d)$ of two orthogonal groups
 naturally acts on it from the right.
\begin{align*}
\pfbM_x &= \{u= \xi\oplus \eta 
\colon \R^d \oplus  \R^{n-d} =\R^n\to T_x\cM
 = \cD_x \oplus  \cD^{\perp}_x \mid
\\
& \qquad\quad
\mbox{$\xi \colon \R^d \to \cD_x$ 
and  $\eta \colon \R^{n-d} \to \cD^{\perp}_x$ are linear isometries}
\},
\\
\pfbM &= \bigsqcup_{x\in \cM} \pfbM_x.
\end{align*}
This is a subbundle of the orthonormal frame bundle 
over the Riemannian manifold $(\cM, \hat{g})$.
For notational simplicity we will write $\cP := \pfbM$ and 
$G:= O (d) \times O(n-d)$.
The Lie algebra of $G$ is ${\frak o} (d) \times {\frak o} (n-d)$,
which will be denoted by ${\frak g}$.
Here, ${\frak o} (d)$ stands for the set of real $d \times d$ 
anti-symmetric matrices. 
In the usual way, we view $G$ as a subgroup of $O(n)$
via the following injection:
\[
G \ni \quad (U, V) \mapsto 
\left(
\begin{array}{@{\,}c|c@{\,}}
 U & O\\  \hline
 O& V\\
\end{array}
\right) 
\quad \in O(n).
\]
We denote the natural projection by $\pi \colon \cP \to \cM$. 
The right action on $\cP$ by $a \in G$ is denoted by $R_a$.
The vertical vector (field) associated with $X \in {\frak g}$ 
is denoted by $X^*$, which is defined by 
$X^* (u) =  (d/dt)|_{t=0} R_{\exp (tX)} (u)$ at 
$u\in \cP$.

Let $U \subset \cM$ be a local coordinate neighborhood 
with a local chart $(x^1, \ldots, x^n)$ and
let $\{Z_1, \ldots, Z_d\}$ and $\{Z_{d+1}, \ldots, Z_n\}$ be 
a local orthonormal frame over $U$
of $\cD$ and $\cD^{\perp}$, respectively.
The canonical orthonormal basis of $\R^n$ is denoted by
$\{ \mathbf{e}_1, \ldots, \mathbf{e}_n\}$.
For $x \in U$,  set $\sigma_U (x)\oplus \tilde\sigma_U (x) \in \cP_x$ by 
$\sigma_U (x) (\mathbf{e}_i) =Z_i (x)$ for all $1 \le i \le d$
and $\tilde\sigma_U (x) (\mathbf{e}_i) =Z_i (x)$ for all $d+1 \le i \le n$.
Then, for every $\xi\oplus \eta \in \cP_x$, there is a unique
$e=\{e_{\alpha\beta}\}_{1 \le \alpha,  \beta \le n} 
  \in G$ such that 
\[
\xi\oplus \eta= (\sigma_U (x)\oplus \tilde\sigma_U (x)) \circ e
= 
\bigl(\sigma_U (x)  \circ \{e_{\alpha\beta}\}_{1 \le \alpha,\beta \le d}
\bigr) \oplus \bigl(
 \tilde\sigma_U (x) \circ 
  \{e_{\alpha \beta}\}_{d+1 \le \alpha, \beta \le n}
\bigr).
\]
Of course, $e_{\alpha\beta} =0$ if $(\alpha, \beta)$ belongs to
\begin{equation}\label{cond.bg}
(\{1, \ldots, d\} \times \{d+1, \ldots, n\})
 \cup (\{d+1, \ldots, n\} \times \{1, \ldots, d\}).
\end{equation}
In this way, we can identify $\pi^{-1} (U) \cong 
U \times G$ with a local chart
$(x^1, \ldots, x^n; \{e_{\alpha\beta}\}_{1 \le \alpha,\beta\le n})$.

Next,  set $\hat\omega = \{\hat\omega^{\alpha}_{\beta}\}_{1 \le \beta, \gamma \le n}$ by 
\begin{equation}\label{def.omega}
\nabla Z_\beta 
=\sum_{\alpha =1}^n
 \hat\omega^{\alpha}_{\beta} Z_{\alpha}
\quad \in \Gamma (U; T^*\cM \otimes T\cM), 
\qquad 1 \le \beta \le n.
\end{equation}
Since $\tilde\nabla$ and $\tilde\nabla^{\perp}$ are metric
and $\nabla$ is defined by \eqref{def.nabla},  we can easily see that
$\hat\omega_{\alpha}^{\beta}=- \hat\omega^{\alpha}_{\beta}$
for all $(\alpha, \beta)$ and 
$\hat\omega^{\alpha}_{\beta}= 0$ if $(\alpha, \beta)$ belongs to
\eqref{cond.bg}.
Hence, $\hat\omega$ is an ${\frak g}$-valued 
one-form on $U$, i.e. $\hat\omega \in
 \Gamma (U;  T^*\cM \otimes {\frak g} )$.
Unfortunately, $\hat\omega$ does not define a global one-form on $\cM$.
We need to lift it to obtain a globally defined one-form on the principal bundle $\cP$.

Define 
$\omega \in \Gamma (U \times G; 
T^*\cP \otimes {\frak g})$ 
     by
\begin{equation}\label{def.bigomega}
\omega = e^{-1} \hat\omega e +  e^{-1} de.
\end{equation}
This is a ${\frak g}$-valued one-form on $\pi^{-1} (U)$.
Note that $e^{-1} de$ stands for Maurer-Cartan form 
(i.e. the left translation of tangent vectors
to the unit element)
on $G$.
By a standard argument, we can see that 
$\omega$ actually defines a global one-form on $\cP$,
namely, we have  
$\omega \in \Gamma (\cP; T^*\cP \otimes {\frak g})$.
It is easy to see that $R_a^* \omega = a^{-1} \omega a$ 
for all $a \in G$ and that $\omega (X^*) = X$ for all $X \in {\frak g}$.
Thus, $\omega$ is an  Ehresmann connection on $\cP$.

On the above coordinate chart $\pi^{-1} (U)$,
we set $\Gamma^{\alpha}_{\beta\gamma}$ by
$\Gamma^{\alpha}_{\beta\gamma} =
\hat\omega^{\alpha}_{\beta} \la Z_{\gamma}\ra$, or equivalently,
\[
\nabla_{Z_{\gamma}}  Z_\beta 
=\sum_{\alpha =1}^n
\Gamma^{\alpha}_{\beta\gamma} Z_{\alpha},
\qquad 
1 \le \alpha, \beta, \gamma\le n.
 \]
We can easily see  that 
$\Gamma^{\alpha}_{\beta\gamma}=0$ if $(\alpha, \beta)$ belongs to \eqref{cond.bg} and that 
$\Gamma^{\alpha}_{\beta\gamma} = 
-\Gamma^{\beta}_{\alpha\gamma}$ for all $\alpha, \beta, \gamma$.
For $u=\xi\oplus \eta \in \pi^{-1}(x)$,
we define a linear injection
 $\ell_{u}\colon T_x \cM \to T_{u}\cP$ by
\begin{equation}\label{def.horlift}
\ell_{u} \la \sum_{\gamma =1}^n c_\gamma Z_\gamma (x) \ra 
=
\sum_{\gamma =1}^n c_\gamma
\Bigl( Z_\gamma (x) - 
\sum_{\alpha, \beta, \delta =1}^n
\Gamma^{\alpha}_{\beta\gamma} (x) e_{\beta\delta}
\frac{\partial}{\partial e_{\alpha\delta}}
\Bigr),
\quad
(c_1, \ldots, c_n) \in \R^n
\end{equation}
and set 
$\cK_{u} = \ell_{u} \la \cD_x \ra$
and 
$\cK_{u}^\perp = \ell_{u} \la \cD^\perp_x \ra$.
It is known that 
\[
\ker (\omega_{u} \colon  
T_{u} \cP \to {\frak g})= \cK_{u}
\oplus \cK^\perp_{u}.
\]
This is called the horizontal subspace
and $\ell_{u}$ is called the horizontal lift.
The subspace
$\cV_{u} = \{ X^*(u) \mid X \in {\frak g}\}$
is called the vertical subspace and it holds that 
$T_{u} \cP= \ker \omega_{u} \oplus \cV_{u}$.
(From these explicit expression, too, we can see that 
the horizontal lift and  the horizontal subspaces 
are compatible with the right action of $G$.)

We define
the canonical horizontal vector fields 
$A_i = \ell_{u} \la u \la \mathbf{e}_i \ra \ra
~(1\le i \le n)$ on $\cP$.
Since $u \la \mathbf{e}_i \ra 
=\sum_{\gamma =1}^n  e_{\gamma i}Z_\gamma (x)$,  
$A_i$ reads
\begin{equation}\label{eq.loc_ex_A}
A_i (u)= \sum_{\gamma =1}^n e_{\gamma i}
\Bigl( Z_\gamma (x) - 
\sum_{\alpha, \beta, \delta =1}^n
\Gamma^{\alpha}_{\beta\gamma} (x) e_{\beta\delta}
\frac{\partial}{\partial e_{\alpha\delta}}
\Bigr)
\end{equation}
in the local chart.
At every $u$, $\{ A_i (u)\}_{i=1}^d$ and  $\{ A_i (u)\}_{i=d+1}^n$
are linear base of $\cK_{u}$ and $\cK_{u}^\perp$,
respectively.
We equip $\cK_{u}\oplus \cK_{u}^\perp$ a unique inner product
so that $\{ A_i (u)\}_{i=1}^n$  becomes an orthonormal basis.
Set a linear isometry
$\theta_u  \colon \cK_{u}\oplus \cK_{u}^\perp \to \R^n$ by 
$\theta  \la A_i (u)\ra =\mathbf{e}_i$ for all $1 \le i \le n$.
We can naturally view $\theta$ as an $\R^n$-valued one-form on $\cP$.
To summarize, the following four bijective linear maps are all isometric.
(In either pair of maps, those in the opposite directions are inverses to each other.)
\[
\R^n =\R^d \oplus \R^{n-d}
\,\,\overset{\theta_u^{-1}}{\underset{\theta_u}{\rightleftarrows}}
 \,\, 
\ker \omega_{u} =\cK_{u}\oplus \cK_{u}^\perp
\,\,\overset{(\pi_*)_u}{\underset{\ell_u}{\rightleftarrows}}
 \,\, 
T_{x}\cM = \cD_{x}\oplus \cD_{x}^\perp.
\]
By restricting this to the first components, we obtain the following isometric correpondence:
\[
\R^d 
\,\,\overset{\theta_u^{-1}}{\underset{\theta_u}{\rightleftarrows}}
 \,\, 
\cK_{u}
\,\,\overset{(\pi_*)_u}{\underset{\ell_u}{\rightleftarrows}}
 \,\, 
\cD_{x}.
\]

\begin{lemma}\label{lem.partHor}
Canonical horizontal vector fields 
$\{A_1, \ldots, A_d\}$ on $\cP$ satisfy the partial H\"ormander 
condition at every $u \in \cP$, that is, 
\[
\pi_* \,
{\rm Lie} (A_1, \ldots, A_d) (u)  = T_{\pi (u)} \cM,
\qquad 
u \in \cP.
\]
\end{lemma}

\begin{proof}
Let  $u \in \pi^{-1} (x)$ and use the local chart as above.
Note that 
\[ 
[ A_1, \ldots, A_d] = [ \ell\la Z_1\ra, \ldots, \ell\la Z_d\ra] E,
\quad
\mbox{where } E= (e_{\alpha\beta})_{1 \le\alpha, \beta \le d}.
\]
Since $E =E(u)$ is a smooth $O (d)$-valued function in $u$,
we have 
\[
{\rm Lie} (A_1, \ldots, A_d) (u)  = 
{\rm Lie} (\ell\la Z_1\ra, \ldots, \ell\la Z_d\ra) (u).
\]
Since we have 
 $\ell\la Z_i\ra (u)= Z_i (x) + \mbox{(a vertical vector field)}$  from \eqref{def.horlift}, 
we see that 
\[
[\ell\la Z_i\ra,  \ell\la Z_j\ra] (u)= [Z_i, Z_j] (x) + \mbox{(a vertical vector field).}
\]
Repeating this, we have
\[
\pi_* \,  {\rm Lie} (\ell\la Z_1\ra, \ldots, \ell\la Z_d\ra) (u)
=
 {\rm Lie} (Z_1, \ldots, Z_d) (x).
\]
Since $\cM$ is sub-Riemannian,
the right hand side equals $T_x \cM$. 
\end{proof}

Now we turn to the (anti-)development of finite energy paths.
Let $\cH^n$ be  Cameron-Martin space over $\R^n$.
For $h \in \cH^n$ and $u\in\cP$, 
we consider the following controlled ODE (skeleton ODE):
\begin{align}\label{eq.odeFB}
d\phi_t = \sum_{i=1}^n  A_i ( \phi_t) dh^i_t, 
\qquad   \phi_0 =u,
\end{align}
and set $\psi_t = \pi (\phi_t)$.
To emphasize the dependency on $h$, we sometimes 
write $\phi(h)$ and $\psi(h)$.
It is clear that $\phi(h) \in \cH_u (\cP, \cK\oplus \cK^\perp)$
and $\psi(h) \in \cH_{\pi(u)} (\cM, T\cM)$.
Moreover, the energy is preserved, that is,
$\| h\|_{\cH^n}^2 = \cE (\phi(h)) = \cE (\psi(h))$.
It should also be noted that $\psi(h)^{\prime}_t  = \phi(h)_t \la h^{\prime}_t\ra$.

The map $h \mapsto \psi (h)$ is bijective.
The inverse of 
\[
h \in \cH^n \mapsto \phi(h) \in \cH_u (\cP, \cK\oplus \cK^\perp)
\]
is given by the line integral of $\theta$:
\[
(u_t)_{t \in [0,1]} \in \cH_u (\cP, \cK\oplus \cK^\perp)
\mapsto \int_0^{\cdot}  \theta_{u_t} \la u^\prime_t\ra   \in \cH^n.
\]

The inverse of the projection 
$\pi\colon \cH_u (\cP, \cK\oplus \cK^\perp) \to \cH_{\pi(u)} (\cM, T\cM)$
is the horizontal lift of paths.
Recall that the horizontal lift of $(x_t)_{t \in [0,1]} 
\in  \cH_{\pi(u)} (\cM, T\cM)$ is a unique 
$(u_t)_{t \in [0,1]} \in \cH_u (\cP, \cK\oplus \cK^\perp)$
such that $u_t^\prime = \ell_{u_t} \la x_t^\prime \ra $ for almost all $t$.
Locally, $(u_t)$ satisfies a simple controlled ODE as follows.
Suppose that $0 <\tau \le 1$ and that
$(x_t)_{t \in [0,\tau]}$ stays in a local chart.
Then, there exists a unique Cameron-Martin path 
$(k_t)_{t \in [0,\tau]}$ over $\R^n$ such that
\[
dx_t = \sum_{i=1}^n Z_i (x_t ) k_t^\prime dt
= \sum_{i=1}^n Z_i (x_t ) d k_t
\quad
\mbox{on $[0,\tau]$ with $x_0 =x =\pi (u)$}.
\]
Therefore, the horizontal lift $(u_t)_{t \in [0,\tau]}$ solves 
the following controlled ODE:
\[
du_t
= \sum_{i=1}^n  \ell \la Z_i \ra (u_t ) d k_t
\quad
\mbox{on $[0,\tau]$ with $u_0 =u$}.
\]
The local expression of $\ell \la Z_i \ra$ was given in \eqref{def.horlift}.
Hence, we can write down this ODE for 
$u_t =(x^1_t, \ldots, x^n_t; \{e_{\alpha\beta, t}\}_{1 \le \alpha,\beta\le n})$
concretely using these coordinates.
(However, we do not elaborate it because it is well-known and
 cumbersome.
The point here is to explain that
the lift map is explicitly computable and dependency on the data 
such as $x, k$ can be tracked.)
Thus, we have  seen that the development map $h \mapsto \phi (h)$
is bijective and preserves energy.

Restricting this correspondence to the first component,
we have the following proposition.
Note that we can naturally view  
$\cH^n = \cH^d \oplus\cH^{n-d}$ as a direct sum of Hilbert spaces.
\begin{proposition}\label{pr.iso_path}
For $h \in \cH^d$ and $u\in\cP$,  consider the following ODE:
\begin{align}\label{eq.odeFB2}
d\phi(h)_t = \sum_{i=1}^d  A_i ( \phi(h)_t) dh^i_t, 
\qquad   \phi(h)_0 =u,
\end{align}
and set $\psi(h)_t = \pi (\phi(h)_t)$.
Then, the development map $h \mapsto \psi(h)$
is an energy-preserving bijection from 
$\cH^d$ to $\cH_{\pi(u)} (\cM, \cD)$.
\end{proposition}


We now provide a generalization of Chow-Rashevsky's theorem 
for future purpose. We will use this in Section \ref{sec.appen}.
\begin{proposition}\label{pr.kunita}
Let $V \in \Gamma (T\cM)$. 
Then, for every $x, y\in \cM$ and $\tau \in (0,1]$, there exists 
$k \in \cH_x (\cM, T\cM)$ such that $k_{\tau} =y$ 
and $k^\prime_t - V(k_t) \in \cD_{k_t} $ for almost all $t \in [0,\tau]$.
\end{proposition}

\begin{proof}
We may simply work on $[0,\tau]$ since we can just set $k$ 
to be constant on $[\tau, 1]$.
Note that if we take $l\in \N_+$ large enough, then
we can find
$B_1, \ldots, B_l \in \Gamma (\cD)$ such that
\[
{\rm Lie} (B_1, \ldots, B_l) (x) = T_x \cM
\qquad   \mbox{for every $x \in \cM$.}
\]
Let us consider the following ODE on $\cM$ controlled by an
$l$-dimensional Cameron-Martin path $h\colon [0,\tau] \to \R^l$:
\[
dk_t = \sum_{i=1}^l  B_i ( k_t) dh^i_t +V( k_t)dt, 
\qquad   k_0 =x.
\]
Thanks to the above condition on  Lie brackets of $B_i$'s, 
this controlled ODE is strongly completely controllable. 
It implies that for every $x, y, \tau$, we can find $h$ such that
the solution $k=k(h)$ satisfies $k_\tau =y$
(see \cite[Section 5]{kun} for example).
Hence, this solution $k$ is a desired path.
\end{proof}


Define a second-order differential operator 
$\tilde\Delta$ on $\cM$ by 
\[
\tilde\Delta f = \mathrm{Trace}_{\cD}  
(\nabla \mathrm{grad}_{\cD} f ),
\qquad
f \in C^2(\cM).
\]
The precise meaning is as follows.
First, $\nabla \mathrm{grad}_{\cD} f \in \Gamma (T^*\cM \otimes \cD)$
and therefore
$v \mapsto \nabla_v \mathrm{grad}_{\cD} f $ can be viewed 
a linear map from $\cD_x$ to itself  at every $x\in \cD$.
The right hand side at $x$ is defined to be
 the trace of this linear map.

\begin{lemma}\label{lem.sa_vect}
Let the notation be as above. Then, 
$
\Delta_{\mathrm{sub}} -\tilde\Delta \in \Gamma (\cD).
$
\end{lemma}

\begin{proof}
Take a local orthonormal frame $\{ Z_1, \ldots, Z_d \}$ of $\cD$
on a coordinate neighborhood $U \subset \cM$.
Then, it is well-known that
\[
\Delta_{\mathrm{sub}}  f= \sum_{i=1}^d 
\{ Z_i^2 f+ (\mathrm{div} Z_i)Z_i f
\},
\]
where $\mathrm{div}$ stands for the divergence with respect to 
the measure $\mathbf{vol}$.

On the other hand, since 
$\mathrm{grad}_{\cD} f= \sum_{i=1}^d (Z_i f)Z_i $, 
we see from \eqref{def.omega} that
\[
\tilde\Delta f= \sum_{i=1}^d  Z_i^2 f+
 \sum_{i,j =1}^d    \hat\omega^j_i(Z_j) Z_i f
 \]
and hence
\[
\Delta_{\mathrm{sub}} - \tilde\Delta 
=
\sum_{i=1}^d  (\mathrm{div} Z_i)Z_i 
- \sum_{i,j =1}^d    \hat\omega^j_i(Z_j) Z_i  \,\, \in \Gamma (U; \cD).
\]
The left hand side is a globally defined at most second-order 
operator.
However, as the right hand side shows, its  second-order part vanishes.
Hence, this is a globally defined first-order 
operator, i.e. a vector field.
\end{proof}

\begin{lemma}\label{lem.op_rel}
Let the notation be as above. Then, we have 
\[
\sum_{i=1}^d A_i^2 (\pi^* f) = \pi^* (\tilde\Delta f),
\qquad
f \in C^2(\cM).
\]
Here, $\pi^* f := f \circ \pi \in  C^2(\cP)$ is the pullback of $f$ by 
the projection $\pi\colon \cP \to \cM$.
\end{lemma}

\begin{proof}
We work with the local chart
$(x^1, \ldots, x^n; \{e_{\alpha\beta}\}_{1 \le \alpha,\beta\le n})$
on $\pi^{-1} (U) \cong U \times G$ as above.
We see from the local expression \eqref{eq.loc_ex_A} that,
for $1 \le i,j \le d$,
\begin{align*}
A_i (\pi^* f) &= \sum_{\gamma =1}^d e_{\gamma i} 
 \pi^* (Z_\gamma f),
 \\
A_jA_i (\pi^* f) &= 
\sum_{\gamma,\delta =1}^d e_{\delta j} e_{\gamma i} 
 \pi^* (Z_\delta Z_\gamma f) 
   - \sum_{\gamma,\delta, \epsilon =1}^d 
       \Gamma^{\gamma}_{\delta\epsilon} 
          e_{\epsilon j}  e_{\delta i} \pi^* (Z_\gamma f).
\end{align*}

On the other hand, noting that
$\{u \la \mathbf{e}_1 \ra, \ldots, u \la \mathbf{e}_d\ra\}$
is an orthonormal basis at $\pi (u)$,
we compute as follows:
\begin{align*}
\bigl\la
\nabla_{u \la \mathbf{e}_j\ra} \mathrm{grad}_{\cD} f, 
u \la \mathbf{e}_i \ra
\bigr\ra_{\cD}
&=
\sum_{\gamma, \delta, \epsilon =1}^d e_{\epsilon j}  e_{\delta i} 
\la
\nabla_{Z_{\epsilon}} (Z_\gamma f) Z_\gamma, 
Z_{\delta} \ra_{\cD}
\\
&=
\sum_{\delta, \epsilon =1}^d e_{\epsilon j}  e_{\delta i} 
Z_{\epsilon}  Z_{\delta} f 
- \sum_{\gamma, \delta, \epsilon =1}^d 
e_{\epsilon j}  e_{\delta i} 
\Gamma^{\gamma}_{\delta\epsilon}  Z_\gamma f
\\
&= A_jA_i (\pi^* f)(u),
\end{align*}
where we used $\Gamma^{\gamma}_{\delta\epsilon}
= - \Gamma^{\delta}_{\gamma\epsilon}$.  
Setting $i=j$ and summing them up, we prove the lemma.
\end{proof}

Now we introduce an SDE on $\cP$.
For $V \in \Gamma (T\cM)$, we set $A_0 \in \Gamma (T\cP)$ by 
\[
A_0 (u) = \ell_u \la V_0 (\pi(u))\ra, \quad u \in \cP, \quad
\mbox{where $V_0 : = V + 
   (\Delta_{\mathrm{sub}} -\tilde\Delta)/2 \in \Gamma(T\cM)$.}
\]
Let $(w_t)_{t \in [0,1]}$ be a standard
 $d$-dimensional Brownian motion
and 
consider the following Stratonovich-type SDE
for $u \in \cP$ and $0< \ve \le 1$:
\begin{equation}\label{eq.sde_on_P}
U_t^{\ve} = \ve \sum_{i=1}^d A_i (U_t^{\ve}) \circ dw^i_t + 
\ve^2 A_0 (U_t^{\ve})dt, \qquad
U_0^{\ve} =u.
\end{equation}
From the scaling property of Brownian motion, the two processes
$(U_t^{\ve})$ and $(U^1_{\ve^2 t})$ have the same law.
We will write $X_t^{\ve} :=\pi(U_t^{\ve})$.

\begin{lemma}\label{lem.sdeP}
Let $x \in \cM$. Choose $u \in \pi^{-1} (x)$ and 
consider SDE \eqref{eq.sde_on_P}. 
Then, the law of 
the process $(X_t^{\ve})_{t \in [0,1]}$ is independent of the choice of $u$
and this process is a diffusion process on $\cM$
associated with the generator 
$\ve^2 (\Delta_{\mathrm{sub}}/2 +V)$.
\end{lemma}

\begin{proof}
By the same reason (the rotational invariance of $(w_t)$)
as in the Riemannian case, the law of 
$(X_t^{\ve})_{t \in [0,1]}$ is independent of the choice of $u$.
By It\^o formula, one can show that
the generator of $(U_t^{\ve})$ is $\ve^2 (\tfrac12 \sum_{i=1}^d A_i^2 + A_0)$.
From this and Lemma \ref{lem.op_rel}, we can easily see that
$(X_t^{\ve})$ is still a diffusion process 
and its generator is $\ve^2 (\Delta_{\mathrm{sub}}/2 +V)$.
\end{proof}

\begin{remark}\label{rem.embed}
For manifold-valued SDEs,  
the manifold need not be a submanifold of a Euclidean space.
However, for manifold-valued RDEs and Malliavin calculus,
the manifold is usually embedded into a Euclidean space.
Therefore, in what follows we choose an embedding 
$\iota\colon \cP \hookrightarrow \R^M$ 
and extend the vector fields $A_i$, $0 \le i \le d$, 
to $C^\infty$ vector fields on $\R^M$ with compact support.
In this way,  we may view SDE \eqref{eq.sde_on_P} 
as an SDE on $\R^M$.
For our purpose, any $M$ and $\iota$ will do.
(Recall that, thanks to the tubular neighborhood theorem, 
every smooth function or vector field on $\cP$
extends to smooth function or vector field
 on $\R^M$ with compact support, respectively.)
\end{remark}

\section{Preliminaries from stochastic analysis}
\label{sec.prob}

In this section we recall several important probabilistic results
which we will use in the proof of our main results.
Basically, all results in this section (except Lemma \ref{lem.dMC_dense}) are either known or easily derived from known facts.

\subsection{Elements of  (manifold-valued) Malliavin calculus}
\label{subsec.MC}
We first recall Watanabe's theory of 
generalized Wiener functionals (i.e. Watanabe distributions) in Malliavin calculus.
Most of the contents and the notations
in this section are contained in 
 \cite[Sections V.8--V.10]{iwbk} with trivial modifications.
Also, \cite{sh, nu, hu, mt} are good textbooks of Malliavin calculus.
For manifold-valued Malliavin calculus, see \cite{ta}.
For basic results of quasi-sure analysis, we refer to \cite[Chapter II]{ma}.

Let $({\cal W}, {\cal H}, \mu)$ be the classical $d$-dimensional
Wiener space.  
(The results in this subsection also hold on any abstract Wiener space, however.)
Let us recall the following:

\begin{enumerate}
\item[{\bf (a)}]
 Basics of Sobolev spaces ${\bf D}_{p,r} ({\cal K})$ of ${\cal K}$-valued 
(generalized) Wiener functionals, 
where $p \in (1, \infty)$, $r \in {\mathbb R}$, and ${\cal K}$ is a real separable Hilbert space.
As usual, we will use the spaces 
${\bf D}_{\infty} ({\cal K})= \cap_{k=1 }^{\infty} \cap_{1<p<\infty} {\bf D}_{p,k} ({\cal K})$, 
$\tilde{{\bf D}}_{\infty} ({\cal K}) 
= \cap_{k=1 }^{\infty} \cup_{1<p<\infty}  {\bf D}_{p,k} ({\cal K})$ of test functionals 
and  the spaces ${\bf D}_{-\infty} ({\cal K}) = \cup_{k=1 }^{\infty} \cup_{1<p<\infty} {\bf D}_{p,-k} ({\cal K})$, 
$\tilde{{\bf D}}_{-\infty} ({\cal K}) = \cup_{k=1 }^{\infty} \cap_{1<p<\infty} {\bf D}_{p,-k} ({\cal K})$ of 
 Watanabe distributions as in \cite{iwbk}.
When ${\cal K} ={\mathbb R}$, we simply write ${\bf D}_{p, r}$, etc.
\item[{\bf (b)}] Meyer's equivalence of Sobolev norms. 
(See \cite[Theorem 8.4]{iwbk}. 
A stronger version can be found in \cite[Theorem 4.6]{sh}.)
\item
[{\bf (c)}] Pullback $T \circ F =T(F)\in \tilde{\bf D}_{-\infty}$ of tempered Schwartz distribution $T \in {\cal S}^{\prime}({\mathbb R}^n)$
on ${\mathbb R}^n$
by a non-degenerate Wiener functional $F \in {\bf D}_{\infty} ({\mathbb R}^n)$. (See \cite[Sections 5.9]{iwbk}.
In fact, this is very strongly related to  Item {\bf (d)} below.)

\item
[{\bf (d)}]A generalized version of integration by parts 
 formula (IbP formula) in the sense 
of Malliavin calculus
 for Watanabe distribution,
 which is given as follows (See \cite[p. 377]{iwbk}):

For 
$F =(F^1, \ldots, F^n) \in {\bf D}_{\infty} ({\mathbb R}^n)$, we denote by 
$\sigma^{ij}_F (w) =  \la DF^i (w),DF^j (w)\ra_{{\cal H}}$
 the $(i,j)$-component of Malliavin covariance matrix ($1 \le i,j \le n$).
We denote by $\gamma^{ij}_F (w)$ the $(i,j)$-component of the inverse matrix $\sigma^{-1}_F$ (if the inverse exists in a certain sense).
Recall that $F$ is called  non-degenerate in the sense of Malliavin
if $(\det \sigma_F)^{-1} \in  \cap_{1<p< \infty} L^p$.
Note that $\sigma^{ij}_F \in {\bf D}_{\infty}$ and
$D \gamma^{ij}_F = -\sum_{k,l} \gamma^{ik}_F ( D\sigma^{kl}_F ) \gamma^{lj}_F $.
Hence, derivatives of $\gamma^{ij}_F$ can be written in terms of
$\gamma^{ij}_F$'s and the derivatives of $\sigma^{ij}_F$'s,
which implies that $\gamma^{ij}_F \in {\bf D}_{\infty}$, too.

Suppose 
$G \in {\bf D}_{\infty}$ and $T \in {\cal S}^{\prime} ({\mathbb R}^n)$.
Then, the following integration by parts holds:
\begin{align}
{\mathbb E} \bigl[
\partial_i T(F ) \cdot G 
\bigr]
=
{\mathbb E} \bigl[
T (F ) \cdot \Phi_i (\, \cdot\, ;G)
\bigr],
\label{ipb1.eq}
\end{align}
where $\Phi_i (w ;G) \in  {\bf D}_{\infty}$ is given by 
\begin{align}
\Phi_i (w ;G) &= \sum_{j=1}^d D^* 
\left( \gamma^{ij }_F \cdot G \cdot DF^j  \right) (w).
\label{ipb2.eq}
\end{align}
Note that ${\mathbb E} $ on the both sides of 
 (\ref{ipb1.eq}) is in fact  the generalized expectation, that is,
the pairing of $\tilde{{\bf D}}_{- \infty}$ and $\tilde{{\bf D}}_{\infty}$.
Here, $D$ and $D^*$ are the $\cH$-derivative
(i.e. the gradient operator in the sense of Malliavin calculus) 
and its adjoint (i.e. the divergence operator).

\item[{\bf (e)}]
If $\eta \in {\bf D}_{-\infty}$ satisfies that
$\la \eta, F\ra \ge 0$ for every non-negative $F \in {\bf D}_{\infty}$,
it is called a positive Watanabe distribution.
According to Sugita's theorem \cite{su}, for every positive Watanabe distribution $\eta$,
there uniquely exists a finite Borel measure $\mu_{\eta}$ on $\cW$
such that
\[
\la \eta, F\ra = \int_{\cW} \tilde{F} (w)  \mu_{\eta} (dw), 
\qquad F \in {\bf D}_{\infty}
\]
holds,
where $\tilde{F}$ stands for $\infty$-quasi-continuous modification of 
$F$.
If $\eta \in {\bf D}_{p, -k}$ is positive, then it holds that 
\[
\mu_{\eta}(A) \le \| \eta \|_{p,-k} {\rm Cap}_{q,k} (A)
\qquad
\mbox{for every Borel subset $A\subset \cW$,}
\]
where $p, q  \in (1, \infty)$ with $1/p +1/q =1$, $k \in \N_+$,
and ${\rm Cap}_{q,k}$ stands for the $(q,k)$-capacity
associated with ${\bf D}_{q,k}$.
(For more details, see \cite[Chapter II]{ma}.)
\end{enumerate}

%
%
%

We will also use a localized version of the Watanabe
distribution theory,
which can be found in  \cite[pp. 216--217]{tw}.
(For proofs, see \cite[Propositions 3.1 and 3.2]{inatan}.)

Let $\rho >0$, $\xi \in {\bf D}_{\infty}$ and $F \in {\bf D}_{\infty} ({\mathbb R}^n)$
and suppose that
\begin{equation}\label{wat1.eq}
\inf_{v \in {\mathbb S}^{n-1}}    v^* \sigma_F v  \ge \rho
\qquad
\mbox{on  \quad$\{w \in {\cal W} \mid |\xi (w)| \le 2 \}$,}
\end{equation}
where ${\mathbb S}^{n-1}$ is the unit ball of ${\mathbb R}^n$.
Let $\chi:{\mathbb R} \to {\mathbb R}$ be a smooth function 
whose support is contained in $[-1, 1]$.
Then, the following proposition holds (see \cite[Proposition 6.1]{tw}).

\begin{proposition}\label{pr.comp1}
Assume (\ref{wat1.eq}).
For every $T \in {\cal S}^{\prime}({\mathbb R}^n)$, 
$\chi(\xi) \cdot T (F) \in  \tilde{\bf D}_{-\infty}$
can be defined in a unique way so that the following properties hold:
\\
\noindent
{\rm (i)}~If $T_k \to T \in {\cal S}^{\prime}({\mathbb R}^n)$ as $k \to \infty$, 
then $\chi(\xi) \cdot T_k (F) \to \chi(\xi) \cdot T (F) \in \tilde{\bf D}_{-\infty}$.
\\
\noindent
{\rm (ii)}~If $T$ is given by $g \in {\cal S}({\mathbb R}^n)$,
then $\chi(\xi) \cdot T (F) = \chi(\xi)  g (F) \in {\bf D}_{\infty}$.
\end{proposition}

We also provide an asymptotic  theorem. It is 
a very special case of \cite[Proposition 6.2]{tw}.
Let $\{ F_{\ve}\}_{0\le \ve \le 1} \subset {\bf D}_{\infty} ({\mathbb R}^n)$
and $\{ \xi_{\ve}\}_{0\le \ve \le 1} \subset {\bf D}_{\infty}$
 be families
of Wiener functionals such that the following asymptotics hold:
\begin{eqnarray}
 F_{\ve}  
 &=&  
 F_0 + O (\ve)
 \qquad
 \qquad \mbox{in ${\bf D}_{\infty} ({\mathbb R}^n)$ as $\ve \searrow 0$,}
  \label{wat2.eq}
  \\
 \xi_{\ve}  
 &=&  
 \xi_0 + O (\ve)
  \qquad
 \qquad \mbox{in ${\bf D}_{\infty} $ as $\ve \searrow 0$.}
  \label{wat3.eq}
 \end{eqnarray}
Here, $O (\ve)$ is the large Landau symbol.
Recall that ${\bf D}_{\infty}$ and ${\bf D}_{\infty} ({\mathbb R}^n)$
are endowed with a natural topology as Fr\'echet spaces.

\begin{proposition}\label{pr.comp2}
Assume (\ref{wat2.eq}), (\ref{wat3.eq}) and $|\xi_0| \le 1/8$.
Moreover, assume that there exists $\rho >0$ independent of $\ve$ such that 
(\ref{wat1.eq}) with $F = F_{\ve}$ and $\xi = \xi_{\ve}$ holds for 
every  $\ve \in (0,1]$.
Let $\chi:{\mathbb R} \to {\mathbb R}$ be a smooth function 
whose support is contained in $[-1, 1]$ such that $\chi (x) =1$ if $|x| \le 1/2$.
Then, we have 
\begin{eqnarray}
\lim_{\ve \searrow 0}
 \chi(\xi_{\ve}) \cdot T (F_{\ve}) = T(F_0)
  \qquad \mbox{in $\tilde{\bf D}_{-\infty}$.}
 \nn
 \end{eqnarray}
More precisely, there exists $k \in \N_+$ such that
the above convergence takes place in ${\bf D}_{p, -k}$ for every $p \in (1,\infty)$.
\end{proposition}

Let us quickly review manifold-valued Malliavin calculus.
Malliavin calculus for SDEs on manifolds was 
developed by Taniguchi \cite{ta}.
Roughly speaking, under suitable assumptions,
almost all of important results 
in the Euclidean case still hold true in the manifold case
with natural modifications.

Let  ${\cal N}$ be a compact manifold
of dimension $m$, which is 
equipped with a smooth volume  $\mathbf{vol}_{{\cal N}}$.
(A measure on ${\cal N}$ is said to be a smooth volume
if it is expressed on each coordinate chart as a strictly positive
smooth density function times the Lebesgue measure.)
Choose a Riemannian metric on $\cN$
so that the determinant of the (determinisitic) Malliavin covariance 
of $\cN$-valued functionals are well-defined.
Any choice of the Riemannian metric
and the smooth volume will do.

An $\cN$-valued Wiener functional 
$F\colon \cW \to \cN$ is said to belong to ${\bf D}_{p,k} ({\cal N})$,
$p \in (1,\infty)$ and $k \in \N+$,
if $f (F) \in {\bf D}_{p,k}$ for every $f \in C^{\infty} (\cN)$.
If $\iota\colon \cN \to\R^M$ is an embedding, then 
$F \in {\bf D}_{p,k} ({\cal N})$ holds if and only if 
$\iota(F) \in {\bf D}_{p,k} (\R^M)$ since every $f \in C^{\infty} (\cN)$
extends to a smooth function on $\R^M$ with compact support.
The same holds true for 
$F \in {\bf D}_\infty ({\cal N}) :=\cap_{k=1 }^{\infty} \cap_{1<p<\infty} {\bf D}_{p,k} ({\cal N})$.
For $F \in  \cup_{1<p<\infty} {\bf D}_{p,1} (\cN)$,
$D_h F (w) \in T_{F(w)} \cN$.
Hence, the Malliavin covariance $\sigma_F (w)$ in this case is  
a symmetric bilinear form on $T^*_{F(w)} \cN\times T^*_{F(w)}\cN$.
Thanks to the Riemannian metric, $\det \sigma_F (w)$ can still be defined.

One of two main results in \cite{ta} is as follows.
As in the Euclidean case, if $F \in {\bf D}_\infty ({\cal N})$ is non-degenerate in the sense of Malliavin, 
i.e. $(\det \sigma_F )^{-1} \in  \cap_{1<p< \infty} L^p$,
then the composition $T(F) =T\circ F \in \tilde{{\mathbf D}}_{-\infty}$
is well-defined 
as a Watanabe distribution for every distribution $T$ on $\cN$.
Moreover, the law of $F$ on $\cN$ 
has a smooth density $p_F$ function with respect to $\mathbf{vol}_{\cN}$.
In particular, 
$\delta_a ( Y^{\ve}_1)$ is a positive Watanabe distribution and 
$p_F (a) = {\mathbb E}[\delta_a (F) ]$ for every $a\in \cN$.
One should note here that $\delta_a$ and 
$ \delta_a (Y_1^{\ve})$ depend on the choice of  $\mathbf{vol}_{\cN}$.
(Since any other smooth volume can be expressed as
$\widehat{\mathbf{vol} }_{\cN} (dy) =\rho (y) \mathbf{vol}_{\cN}  (dy)$
for some strictly positive smooth function $\rho$ on $\cN$,
the delta function with respect to 
$\widehat{\mathbf{vol}}_{\cN}$ is given by $\hat{\delta}_a =\rho (a)^{-1} \delta_a$.)
The other main result in \cite{ta} is proving non-degeneracy  
for the projected solution of an SDE whose coefficient 
vector fields satisfies the partial H\"ormander condition.

From here we consider SDE \eqref{eq.sde_on_P}
(with $x \in \cM$ and $u\in \pi^{-1} (x)$)
and set $\cN$ to be either $\cM$ or $\cP$.
By Remark \ref{rem.embed}, we can easily see
from  the corresponding result in the Euclidean case 
that $U^{\ve}_t\in {\bf D}_\infty (\cP)$ for every $(\ve, t) \in [0,1]^2$
and  
$X^{\ve}_t=\pi (U^{\ve}_t) \in {\bf D}_\infty (\cM)$
and that, for every $f \in C^{\infty} (\cP)$, $p \in (1,\infty)$, $k \in \N$, 
${\bf D}_{p,k}$-norm
of $f(U^{\ve}_t)$ is bounded in $(\ve, t) \in [0,1]^2$.
In Lemma \ref{lem.partHor}, we checked the partial H\"ormander condition.
Therefore, $X^{\ve}_t$ is non-degenerate for every $t, \ve \in (0,1]$.
More precisely, the following Kusuoka-Stroock's  estimate is known:
There exist a constant $\nu >0$ independent of $p$
and a constant $C_p >0$
such that, for every $1 <p<\infty$, 
\begin{equation}\label{eq.KS_est}
\| (\det \sigma_{X^{\ve}_1 })^{-1}  \|_{L^p}  \le C_p \ve^{-\nu},
\qquad
 \ve\in (0,1].
\end{equation}

Combining this with Lemma \ref{lem.sdeP}, 
the transition probability of
$\ve^2 (\Delta_{\mathrm{sub}}/2 +V)$-diffusion 
has a density $p^{\ve}_t (x, a)$
with respect to $\mathbf{vol} (da)$, which is smooth in $a \in \cM$.
It holds that
\begin{equation}\label{eq.Wat_rep}
p^{\ve}_t (x, a) =p^{1}_{\ve^2 t} (x, a) = 
{\mathbb E}[\delta_a (X^{\ve}_t ) ].
\end{equation}
We will show in Section \ref{sec.appen}
that $p^{\ve}_t (x, a) >0$ for all $\ve, t \in (0,1]$
and $x, a \in \cM$,
which enables us to define the pinned diffusion measure
associated with $\ve^2 (\Delta_{\mathrm{sub}}/2 +V)$
from every $x$ to every $a$.
(We will later make sure that the measure actually exists.)


\subsection{Elements of rough path theory}

In this subsection we recall the geometric rough path space 
with H\"older or Besov norm 
and quasi-sure properties of the rough path lift.
For basic properties of geometric rough path space, 
we refer to \cite{lcl, fvbk}.
For the geometric rough path space with Besov norm, 
we refer to \cite[Appendix A.2]{fvbk}.
The quasi-sure properties of the 
rough path lift is summarized in \cite{in1}.
In this paper we assume  
$\alpha \in (1/3, 1/2)$ for the H\"older parameter.
We also assume that 
the Besov parameter $(\alpha, 4m)$ satisfy  the following 
conditions:
\begin{equation}\label{eq.amam}
\frac13 <\alpha < \frac12, \quad m \in \N_+, \quad
\al - \frac{1}{4m} > \frac13, \quad
 4m (\frac12 -\alpha)   >1.
\end{equation}
We work in Lyons' original formulation of RDEs (see \cite{lcl}),
but we basically study the first level paths of solutions only.
For brevity we will write $\lambda^\ve_t :=\ve^2 t$ for $\ve \in (0,1]$.
%


We denote by  
$G\Omega^H_{\alpha} ( {\mathbb R}^d)$  the $\alpha$-H\"older 
geometric rough path space over ${\mathbb R}^d$.
A generic element of $G\Omega^H_{\alpha} ( {\mathbb R}^d)$
is denoted by ${\bf w} =({\bf w}^1, {\bf w}^2)$.
For $\beta \in (0,1]$,
let $C_0^{\beta-H}([0,1], {\mathbb R}^k)$
be the Banach space of all ${\mathbb R}^k$-valued $\beta$-H\"older continuous paths 
that start at $0$.
If $\alpha + \beta >1$, then the Young pairing 
\[
 G\Omega^H_{\alpha} ( {\mathbb R}^d) \times C_0^{\beta-H}([0,1], {\mathbb R}^k)
  \ni ({\bf w}, \lambda) \mapsto
  ({\bf w}, \bm{\lambda}) 
     \in  G\Omega^H_{\alpha} ( {\mathbb R}^{d+k})
     \]
is a well-defined, locally Lipschitz continuous map.
(See \cite[Section 9.4]{fvbk} for example.)


Now we consider a system of RDEs driven by the Young pairing
 $({\bf w}, \bm{\lambda}) \in  G\Omega^H_{\alpha} ( {\mathbb R}^{d+1})$
of ${\bf w} \in G\Omega^H_{\alpha} ( {\mathbb R}^{d})$
and 
$\lambda \in C_0^{1-H}([0,1], {\mathbb R}^1)$.
(The main example we have in mind is $\lambda_t = \mbox{const} \times t$.)
For vector fields $V_{i}: {\mathbb R}^n \to {\mathbb R}^n$ ($0 \le i \le d$), consider
\begin{equation}\label{rde_x.def}
dx_t = \sum_{i=1}^d  V_i ( x_t) dw_t^i + V_0 ( x_t) d\lambda_t,
\qquad
 x_0 =x \in \R^n.
\end{equation}
The RDEs for the Jacobian process and its inverse are given as follows;
\begin{eqnarray}
dJ_t &=& \sum_{i=1}^d  \nabla V_i ( x_t) J_t dw_t^i + \nabla V_0 ( x_t) J_t d\lambda_t,
\qquad
J_0 ={\rm Id}_n,
\label{rde_J.def}
\\
dK_t &=& - \sum_{i=1}^d K_t  \nabla V_i ( x_t)  dw_t^i  - K_t \nabla V_0 ( x_t) d\lambda_t,
\qquad
K_0 ={\rm Id}_n.
\label{rde_K.def}
\end{eqnarray}
Note that
$J, K,$ and  $\nabla V_i $ are ${\rm Mat}(n,n)$-valued.
Here, ${\rm Mat}(n,m)$ stands for the set of all real $n\times m$ matrices and ${\rm Id}_n$ stands for the identity matrix of size $n$.

For simplicity we
assume that $V_i$, $0 \le i \le d$, is of $C_b^\infty$, 
that is, when viewed as an $\R^n$-valued function,
$V_i$ is a bounded smooth function with bounded derivatives 
of all order.
It is then known that a unique global solution of (\ref{rde_x.def})--(\ref{rde_K.def}) 
exists for any ${\bf w}$ and $\lambda$ . 
Moreover, Lyons' continuity theorem holds.
In that case, the following map is continuous:
\[
 G\Omega^H_{\alpha} ( {\mathbb R}^d) \times C_0^{1-H}([0,1], {\mathbb R}^1)
  \ni ({\bf w}, \lambda) \mapsto
  ({\bf x}, {\bf J}, {\bf K}) \in G\Omega^H_{\alpha} ( {\mathbb R}^n \oplus {\rm Mat}(n,n)^{\oplus 2}).
    \]
The map $({\bf w}, \lambda) \mapsto  {\bf x}$ is denoted by 
$\Phi\colon G\Omega^H_{\alpha} ( {\mathbb R}^d) \times C_0^{1-H}([0,1], {\mathbb R}^1) 
\to G\Omega^H_{\alpha} ( {\mathbb R}^n)$.
(We adopt Lyons' formulation of RDEs as in \cite{lcl}.
So, the initial values of the first level paths 
must be adjusted.)
If $w \in C_0^{1-H}([0,1], {\mathbb R}^d)$ 
or $w \in \cH=\cH^d$
and ${\bf w}$ is its natural lift, 
then the path
\begin{equation}\label{eq.3de}
 t \mapsto ( x + {\bf x}^1_{0,t}, {\rm Id}+{\bf J}^1_{0,t}, {\rm Id}+{\bf K}^1_{0,t} )
\end{equation}
coincides with the solution of a system (\ref{rde_x.def})--(\ref{rde_K.def}) of 
ODEs understood in the usual Riemann-Stieltjes sense.
Recall that ${\bf x}^1_{0,t}$ is the first level path of ${\bf x}$ evaluated at $(0,t)$.
Keep in mind that
$({\rm Id}+{\bf J}^1_{0,t})^{-1}= {\rm Id}+{\bf K}^1_{0,t}$ always holds.

When ${\bf w}$ is Brownian rough path ${\bf W}$
i.e. the natural lift of $(w_t)$, and $\lambda= \lambda^1$, 
the process in
\eqref{eq.3de} coincides $\mu$-a.s. with the corresponding system of 
usual Stratonovich SDEs with drift.
In this case $x_t:=x + {\bf x}^1_{0,t}  \in {\bf D}_\infty (\R^n)$ for every $t$.
If $G\colon \R^n \to \R^m$ is a smooth map 
with bounded derivatives 
of all order ($\ge 1$),  then 
$G(x_t) \in {\bf D}_\infty (\R^m)$ and for every $h \in \cH$ and $t$,
it holds that
\begin{equation}\label{eq.Dxt}
D_h G(x_t) 
= (\nabla G) (x + {\bf x}^1_{0,t})
({\rm Id}+{\bf J}^1_{0,t})
\int_0^t ({\rm Id}+{\bf K}^1_{0,s}) {\bf V} ( x + {\bf x}^1_{0,s}) dh_s,
\quad\mbox{a.s.}
 \end{equation}
(with ${\bf w}= {\bf W}$).
Here, we view ${\bf V} :=[V_1, \ldots, V_d] \in  {\rm Mat}(n,d)$
and $\nabla G \in  {\rm Mat}(m,n)$.


We define a continuous function
\begin{equation}\label{def.Gamma}
\Gamma\colon  G\Omega^H_{\alpha} ( {\mathbb R}^d) \times C_0^{1-H}([0,1], {\mathbb R}^1) \times [0,1]
\to {\rm Mat}(m,m)
\end{equation}
as follows:
Set 
\[
\Gamma ({\bf w}, \lambda)_t 
=(\nabla G) (x + {\bf x}^1_{0,t}) ({\rm Id}+{\bf J}^1_{0,t}) 
C ({\bf w}, \lambda)_t ({\rm Id}+{\bf J}^1_{0,t})^*(\nabla G) (x + {\bf x}^1_{0,t})^*,
\]
where
\[
C ({\bf w}, \lambda)_t
:=
\int_0^t  ({\rm Id}+{\bf K}^1_{0,s}) {\bf V} ( x + {\bf x}^1_{0,s})  
{\bf V}( x + {\bf x}^1_{0,s})^*  ({\rm Id}+{\bf K}^1_{0,s})^* ds.
\]
Here, the superscript $*$ stands for the transpose of a matrix.
Then, \eqref{eq.Dxt} implies that when ${\bf w}= {\bf W}$ and $\lambda=\lambda^1$, 
Malliavin covariance matrix of 
$G(x + {\bf x}^1_{0,t})$ equals $\Gamma ({\bf W}, \lambda^1)_t$, 
$\mu$-a.s.
Similarly for $h \in \cH$, the deterministic 
Malliavin covariance matrix of 
$G(x + {\bf x}^1_{0,t})$ equals $\Gamma ({\bf h}, \lambda^1)_t$,
where ${\bf w}={\bf h}$ is the natural rough path lift of $h$.
(When $G$ is the identity map, these formulas are well-known.
The general  case is just a straightforward modification.)

For $(\al, 4m)$ which satisfies \eqref{eq.amam}, 
$G\Omega^B_{\alpha, 4m} ( {\mathbb R}^d)$ denotes the geometric rough path space 
over ${\mathbb R}^d$ with $(\al, 4m)$-Besov norm.
Recall that the distance on this space is given by 
\begin{align}
d({\bf w}, \hat{\bf w}) 
&= \| {\bf w}^1- \hat{\bf w}^1 \|_{\al, 4m-B} 
+\| {\bf w}^2- \hat{\bf w}^2 \|_{2\al, 2m-B}
\nn\\
&
:=
\Bigl(
\iint_{0 \le s <t \le 1}  \frac{ | {\bf w}^1_{s,t}- \hat{\bf w}^1_{s,t}|^{4m}}
{|t-s|^{1 +4m\al }} 
dsdt
\Bigr)^{1/4m}
+
\Bigl(
\iint_{0 \le s <t \le 1}  \frac{ | {\bf w}^2_{s,t}- \hat{\bf w}^2_{s,t}|^{2m}}
{|t-s|^{1 +4m\al }} 
dsdt
\Bigr)^{1/2m}.
\nn
\end{align}
By the Besov-H\"older embedding theorem for rough path spaces,
there is a continuous embedding $G\Omega^B_{\alpha, 4m} ( {\mathbb R}^d)
 \hookrightarrow G\Omega^H_{\alpha -(1/4m)} ( {\mathbb R}^d)$.
If $\al < \al' <1/2$, there is a continuous embedding
$G\Omega^H_{\alpha'} ( {\mathbb R}^d)
\hookrightarrow G\Omega^B_{\alpha, 4m} ( {\mathbb R}^d)$.
Basically, we will not write the first embedding explicitly.
(For example, if we write $\Phi({\bf w}, \lambda)$ for  
$({\bf w}, \lambda) \in G\Omega^B_{\alpha, 4m} ( {\mathbb R}^d) \times C_0^{1-H}([0,1], {\mathbb R}^1)$,
then it is actually the composition of the first embedding map above and $\Phi$ with respect to 
$\{\al -1/(4m)\}$-H\"older topology.)
It is known that the Young translation by $h \in {\cal H}$ works 
well on $G\Omega^B_{\alpha, 4m} ( {\mathbb R}^d)$ under (\ref{eq.amam}).
The map $({\bf w}, h) \mapsto \tau_h ({\bf w})$
is continuous from $G\Omega^B_{\alpha, 4m} ( {\mathbb R}^d) \times {\cal H}$
to $G\Omega^B_{\alpha, 4m} ( {\mathbb R}^d)$, 
where 
$ \tau_h ({\bf w})$ is the Young translation of ${\bf w}$ by $h$
(see \cite[Lemma 5.1]{in1}).


Now we review quasi-sure properties of rough path lift map 
${\cal L}$ from ${\cal W}$ to $G\Omega^B_{\alpha, 4m} ( {\mathbb R}^d)$.
For $k=\N_+$ and $w \in {\cal W}$,
we denote by $w(k)$ the $k$th dyadic piecewise linear approximation of $w$ 
associated with the partition $\{ j2^{-k} \mid 0 \le j \le 2^k\}$ of $[0,1]$.
We denote the natural lift of $w(k)$ by $ {\cal L} (w(k))$,
which is defined by Riemann-Stieltjes (or Young) integral.
We set 
\[
{\cal Z}_{\al, 4m} := \bigl\{ w \in {\cal W} \mid
\mbox{ $\{ {\cal L} (w(k)) \}_{k=1}^{\infty}$ is Cauchy in $G\Omega^B_{\alpha, 4m} ( {\mathbb R}^d)$} 
\bigr\}.   
\] 
We define ${\cal L}: {\cal W} \to G\Omega^B_{\alpha, 4m} ( {\mathbb R}^d)$
by ${\cal L} (w) = \lim_{m\to \infty} {\cal L} (w(k))$ if $w \in {\cal Z}_{\al, 4m}$
and  we define ${\cal L} (w)$ to be the zero rough path
 if $w \notin {\cal Z}_{\al, 4m}$.
We will use this version of ${\cal L}$, which is Borel measurable,
and write ${\bf W} := {\cal L} (w)$ (as before)
when it is regarded 
as a rough path space-valued random variable defined on ${\cal W}$.

Note that 
$\cH$ and $C_0^{\beta-H}([0,1], {\mathbb R}^d)$ with $\beta \in (1/2, 1]$
are subsets of ${\cal Z}_{\al, 4m}$
and the two definition of rough path lift coincide.
(We will often write ${\bf h}={\cal L} (h)$ for $h \in \cH$.)
Under the scalar multiplication (i.e. the dilation)
and the Cameron-Martin translation,  
 ${\cal Z}_{\al, 4m}$ is left invariant.
Moreover, $c {\cal L} (w) = {\cal L} (cw)$ and $\tau_h({\cal L} (w))= {\cal L} (w+h)$ 
for any $w \in {\cal Z}_{\al, 4m}$, $c \in \R$, and $h \in {\cal H}$. 
It is known that ${\cal Z}_{\al, 4m}^c$ is slim, that is,
 the $(p,r)$-capacity of 
this set is zero for any $p \in (1,\infty)$ and $r \in {\mathbb N}_+$.
Therefore, from a viewpoint of quasi-sure analysis, 
the lift map ${\cal L}$ is well-defined.
Moreover, 
the map ${\cal W} \ni w \mapsto {\cal L} (w) \in G\Omega^B_{\alpha, 4m} ( {\mathbb R}^d) $
is $\infty$-quasi-continuous.
(This kind of $\infty$-quasi-continuity was first shown in \cite{aida}.)

Let $\ve \in (0,1]$ be a small parameter.
We  now
recall that the unique solution of an RDE
driven by $\ve {\bf W}= {\cal L} (\ve w)$ 
coincides with that of the corresponding scaled Stratonovich SDE
given as follows:
\begin{equation}\label{sc_sde.def}
dX^{\ve}_t = \ve \sum_{i=1}^d  V_i ( X^{\ve}_t) \circ  dw_t^i  + \ve^2  V_0 (X^{\ve}_t)   dt,
\qquad
 X^{\ve}_0 =x \in {\mathbb R}^n.
\end{equation}
When necessary, we will write $X^{\ve}_t = X^{\ve}(t, x, w)$ or $X^{\ve}(t, x)$.
Then, $X^{\ve} (\,\cdot\, , x,w) = x + \Phi(\ve {\bf W}, \lambda^\ve)^1$ holds $\mu$-a.s., which means that 
 the right hand side is an $\infty$-quasi-continuous modification
of the left hand side as a 
$C^{\alpha -H}([0,1], {\mathbb R}^n)$-valued Wiener functional.
Similarly, for every $t, \ve \in (0,1]$,
$\Gamma (\ve {\bf W}, \lambda^\ve)_t 
= \ve^{-2}\sigma_{X_t^{\ve}}$ holds, $\mu$-a.s.
Therefore, not just $X^{\ve}_t$ itself, but also
its Malliavin covariance matrix $\sigma_{X_t^{\ve}}$
is a continuous function of Brownian rough path.

The skeleton ODE (without drift) associated with 
small noise problems for the above SDE
 \eqref{sc_sde.def}  is given as follows:  
 For $h \in \cH$,
\begin{equation}\label{sc_skel.def}
d\phi_t =  \sum_{i=1}^d  V_i ( \phi_t)  dh_t^i,
\qquad
\phi_0 =x \in {\mathbb R}^n.
\end{equation}
We write the unique solution $\phi = \phi (h)$ when necessary
(which equals $x + \Phi({\bf h}, 0)^1$).
The deterministic Malliavin covariance matrix of 
$\phi (h)_t$ at $h$ is denoted by $\sigma_{\phi_t}(h)$.
As is well-known, 
$\Gamma ({\bf h}, 0)_t = \sigma_{\phi_t}(h)$.
Recall that $\det \sigma_{\phi_t}(h) >0$ if and only if 
the tangent map of $\phi_t \colon \cH \to \R^n$ at $h$ is surjective.

Set $X^{\ve, h} := X^{\ve}(\,\cdot\,, x, w+(h/\ve)) 
= x + \Phi(\tau_h (\ve {\bf W}), \lambda^\ve)^1$ 
for $h \in \cH$.
In other words, $X^{\ve, h}$ uniquely solves the following 
Stratonovich SDE:
\[
dX^{\ve, h}_t =  \sum_{i=1}^d  V_i ( X^{\ve}_t) \circ  d (\ve w_t^i  + h_t)
+\ve^2  V_0 (X^{\ve}_t)   dt,
\qquad
 X^{\ve, h}_0 =x.
 \]
Small noise asymptotics of $X^{\ve, h}$ has been extensively studied.
One of basic results is the following asymptotics.
\[
X^{\ve, h}_1 = \phi (h)_1 +  \ve \Xi^h_1 + O (\ve^2) 
 \qquad \mbox{in ${\bf D}_{\infty} ({\mathbb R}^n)$ 
 as $\ve\searrow 0$},
 \]
where $\Xi^h_1$ is the element of the first order Wiener chaos
given by the following Wiener integral:
\[
\Xi_t^h (w) =  
({\rm Id}+{\bf J}({\bf h}, 0)^1_{0,t})
\int_0^t ({\rm Id}+{\bf K}({\bf h}, 0)^1_{0,s}) {\bf V} ( \phi (h)_s) dw_s.
\]
Hence, for $G\colon \R^n \to \R^m$ as above, 
\[
G(X^{\ve, h}_1) 
= G(\phi (h)_1 ) +  \ve   (\nabla G) (\phi (h)_1)  \Xi^h_1 + O (\ve^2) 
 \qquad \mbox{in ${\bf D}_{\infty} ({\mathbb R}^m)$ 
 as $\ve\searrow 0$}.
\]
Note that $(\nabla G) (\phi (h)_1)  \Xi^h_1$
belongs to the first order Wiener chaos
and therefore induces a mean-zero Gaussian measure on $\R^m$.
Its covariance matrix equals 
$\Gamma ({\bf h}, 0)_1$, which in turn equals  
the  deteministic 
Malliavin covariance matrix of $\cH\ni k \mapsto G(\phi (k)_1)$
at $h$.

The skeleton ODE (with drift) associated with the above SDE
 \eqref{sc_sde.def} with $\ve =1$ is given as follows:  
 For $h \in \cH$,
\begin{equation}\label{sc_skel.def}
d\zeta_t =  \sum_{i=1}^d  V_i ( \zeta_t)  dh_t^i  
+ V_0 (\zeta_t)   dt,
\qquad
\zeta_0 =x \in {\mathbb R}^n.
\end{equation}
We write the unique solution $\zeta = \zeta (h)$ when necessary
 (which equals $x + \Phi({\bf h}, \lambda^1)$).
The deterministic Malliavin covariance matrix of 
$\zeta (h)_t$ at $h$ is denoted by $\sigma_{\zeta_t}(h)$.
As is well-known, $\Gamma ({\bf h}, \lambda^1)_t = \sigma_{\zeta_t}(h)$.
\begin{lemma}\label{lem.dMC_dense}
Consider SDE \eqref{sc_sde.def} with $\ve =1$
and ODE \eqref{sc_skel.def}. 
We assume that at $t\in (0,1]$,
$\det \sigma_{X_t^1} >0$ holds,  $\mu$-a.s.
Then, 
$
\{ {\bf h}=\cL (h) \mid h\in \cH, \det \sigma_{\zeta_t}(h) >0\}
$
is dense in $G\Omega^H_{\alpha} ( {\mathbb R}^d)$
for any $1/3 <\alpha <1/2$.
\end{lemma}

\begin{proof}
Take $m \in \N_+$ so large that $(\al +(1/4m), 4m)$ still
satisfies \eqref{eq.amam} and set 
\[
A:=\{ w \in  {\cal Z}_{\al +(1/4m), 4m} \mid \det \sigma_{X_t^{\ve}} (w)>0\} \subset \cW.
\]
This subset is of full $\mu$-measure
and hence $\cL (A)$ is of full measure with respect to the law of 
Brownian rough path.
Note that $\cL (A)\subset G\Omega^H_{\alpha} ( {\mathbb R}^d)$
due to the Besov-H\"older embedding theorem.
Thanks to the support theorem for Brownian rough path
(see \cite[Theorem 13.54]{fvbk}), $\cL (A)$ must be dense in 
$G\Omega^H_{\alpha} ( {\mathbb R}^d)$.
For every $w \in A$, we have $\lim_{k\to\infty}\cL (w(k)) = \cL (w)$
in $\alpha$-H\"older topology and therefore
\[
\lim_{k\to\infty}\sigma_{\zeta_t}(w(k)) 
= 
\lim_{k\to\infty} \Gamma (\cL(w(k)), \lambda)_t
= 
\Gamma (\cL(w), \lambda)_t
=\sigma_{X_t^1} (w).
\]
This implies that $\det \sigma_{\zeta_t}(w(k)) >0$ for large enough $k$.
This proves the lemma.
\end{proof}

%
\section{Large deviations for rough path lift of positive
Watanabe distributions}
\label{sec.LDP_wd}

In this section we formulate an LDP
 for the rough path lifts of Watanabe's pull-back of 
the delta functions, from which our main theorem 
(Theorem \ref{thm.mainQ}) easily follows.

For $x \in \cM$, we take any 
$u \in \pi^{-1} (x)$ and consider SDE \eqref{eq.sde_on_P}
driven by the canonical realization of 
$d$-dimensional Brownian motion $(w_t)_{t\in [0,1]}$.
By \eqref{eq.KS_est},
$\delta_a (X^{\ve}_t) \in \tilde{\bf D}_{-\infty}$ 
is a well-defined positive Watanabe distribution for every 
$a\in\cM$.
By the positivity of the heat kernel (which will be proved in Section \ref{sec.appen}),
we see that
$p_t^{\ve}(x,a)={\mathbb E}[ \delta_a (X^{\ve}_t)] >0$ for all $t, \ve \in (0,1]$ and $x, a \in \cM$. 
By Sugita's theorem \cite{su}, 
the 
positive Watanabe distribution $\delta_a (X^{\ve}_1)$ at time $t=1$
is in fact a non-trivial finite Borel measure on ${\cal W}$,
which will be denoted by $\theta^{\ve}_{u,a}$.

Since 
${\cal L}$ is defined outside a slim set, we can lift 
the measure $\theta^{\ve}_{u,a}$ to a
measure on $G\Omega^B_{\alpha, 4m} ( {\mathbb R}^d)$.
We write 
$\nu^{\ve}_{u,a} =(\ve  {\cal L})_* [\theta^{\ve}_{u,a}]$.
Here, $\ve  {\cal L}$ is the composition of $\cL$ and
the dilation by $\ve$.
Since the complement of $\cZ_{\alpha, 4m}$ is slim, 
$\nu^{\ve}_{u,a}$ does not depend on how $\cL$ is defined 
on this complement.
We denote by $\hat\theta^{\ve}_{u,a}$ and $\hat\nu^{\ve}_{u,a}$
the normalized measure of 
$\theta^{\ve}_{u,a}$ and $\nu^{\ve}_{u,a}$, respectively.
(Since the total mass of $\theta^{\ve}_{u,a}$ or of $\nu^{\ve}_{u,a}$ 
equals ${\mathbb E}[ \delta_a (X^{\ve}_t)] >0$, this normalization is well-defined.)

Let $\phi (h)$ be the solution of ODE \eqref{eq.odeFB2} 
and write $\psi (h)= \pi (\phi (h))$.  
In what follows, we write $\cH = \cH^d$ for simplicity.
We set
\begin{equation}\label{def.Qua}
{\cal Q}^{u, a} =\{ h \in \cH \mid \psi(h)_1 =a\}.
\end{equation}
By Chow-Rashevsky's theorem 
\cite[Theorem 1.14]{ri},
there exists an admissible path 
$(x_t)_{t\in [0,1]} \in \cH_{x} (\cM, \cD)$
such that $x_1 =a$.
By Proposition \ref{pr.iso_path},
its anti-development belongs to ${\cal Q}^{u, a}$.
This implies that
${\cal Q}^{u, a} \neq \emptyset$ for any $u$ and $a$.

Define a rate function $I_1 : G\Omega^B_{\al, 4m} ({\mathbb R}^n) \to [0, \infty]$ as follows;
\begin{align}
I_1 ({\bf w}) 
= 
\begin{cases}
  \tfrac12  \|h\|^2_{\cH} & (\mbox{if ${\bf w}= {\cal L}(h)$ for some $h \in {\cal Q}^{u,a}$}), \\
    \infty &  (\mbox{otherwise}).
  \end{cases}
\nn
\end{align}
From the Schilder-type LDP for Brownian rough path \cite[Theorem 13.42]{fvbk},
we can easily see that $I_1$ is good.
Also define $\hat{I}_1  ({\bf w}) 
= I_1 ({\bf w}) - \min\{ \|h\|^2_{\cH}/2  \mid h \in {\cal Q}^{u,a}\}$,
which is also good.
From the goodness of $I_1$ and Proposition \ref{pr.iso_path}, 
we can easily see that
the minimum above exists and equals $d_{SR} (x,a)^2/2$.

Our main purpose of this section is to prove that
$\{\nu^{\ve}_{u,a} \}_{0 <\ve \le 1}$
satisfies 
an LDP of Schilder type on $G\Omega^B_{\alpha, 4m} ( {\mathbb R}^d)$
as $\ve \searrow 0$.
As we will see, our main result  easily follows from 
the following theorem.

\begin{theorem}\label{tm.ldp.dlt}
Let the notation as above.
Let $u \in \cP$ and $a \in \cM$ and 
assume (\ref{eq.amam}) for the Besov parameter $(\alpha, 4m)$.
Then, the following {\rm (i)} and {\rm (ii)} hold:
\\
\noindent
{\rm (i)}~ The family $\{ \nu^{\ve}_{u,a}\}_{0<\ve \le 1}$ 
of finite measures
is exponentially tight and  satisfies an LDP 
on $G\Omega^B_{\al, 4m} ({\mathbb R}^d)$ as $\ve \searrow 0$ 
with the speed $\ve^2$ and
a good rate function $I_1$, that is, 
for every Borel set $A \subset G\Omega^B_{\al, 4m} ({\mathbb R}^d)$,
the following inequalities hold;
\begin{align}
- \inf_{{\bf w} \in A^{\circ} } I_1 ({\bf w})  
\le
 \liminf_{\ve \searrow 0 } \ve^2 \log \nu^{\ve}_{u,a} (A^{\circ})
\le
 \limsup_{\ve \searrow 0 } \ve^2 \log \nu^{\ve}_{u,a} (\bar{A})
\le 
- \inf_{{\bf w} \in \bar{A} } I_1  ({\bf w}).
\nn
\end{align}
\noindent
{\rm (ii)}~ 
The family $\{ \hat\nu^{\ve}_{u,a}\}_{0 < \ve\le 1}$ 
of probability measures 
is exponentially tight and satisfies an LDP 
on $G\Omega^B_{\al, 4m} ({\mathbb R}^d)$ as $\ve \searrow 0$ 
with the speed $\ve^2$ and a good rate function $\hat{I}_1$.
\end{theorem}

Since the whole set is both open and closed,
Theorem \ref{tm.ldp.dlt} {\rm (i)}  implies that 
\begin{align*}
\lim_{t \searrow 0} t  \log p^{1}_{ t} (x, a)
&= \lim_{\ve \searrow 0 } \ve^2 \log {\mathbb E}[ \delta_a (X^{\ve}_1)] 
\\
&=
\lim_{\ve \searrow 0 } \ve^2 \log \theta^{\ve}_{u,a} (\cW ) 
\\
&=
\lim_{\ve \searrow 0 } \ve^2 \log \nu^{\ve}_{u,a} ( G\Omega^B_{\al, 4m} ({\mathbb R}^d) ) 
\\
&= 
- \min\{ \|h\|^2_{{\cal H}}/2  \mid h \in {\cal Q}^{x,a}\}
=
- d_{SR} (x,a)^2/2.
\end{align*}
We have also used \eqref{eq.Wat_rep} above.
Due to this Varadhan-type asymptotic formula, 
Theorem \ref{tm.ldp.dlt} {\rm (ii)} is immediate from {\rm (i)}.
We will prove Theorem \ref{tm.ldp.dlt} {\rm (i)} in 
Sections \ref{sec.lower} and \ref{sec.upper}.

\begin{proof}[Proof of Theorem \ref{thm.mainQ}.]
Now we consider $G\Omega^H_{\alpha} ( {\mathbb R}^d)$
with $\alpha \in (1/3, 1/2)$.
Due to the Besov-H\"older embedding, $\hat\nu^{\ve}_{u,a}$
actually sits on this space
and Theorem \ref{tm.ldp.dlt} {\rm (ii)} still holds 
even if $G\Omega^B_{\al, 4m} ({\mathbb R}^d)$ is replaced 
by this space.
Obviuosly, the family $\{\delta_{\lambda^\ve} \}_{0<\ve \le 1}$
is exponentially tight and 
satisfies an LDP on $C_0^{1-H}([0,1], {\mathbb R})$ with
the good rate function $+\infty \cdot {\bf 1}_{\{ 0\}^c}$
with the convention that $\infty \cdot 0 =0$.
By a general result for LDPs for product measures
(see \cite[Exercise 4.2.7]{dzbk}), 
in which the exponential tightness plays a key role,
$\{\hat\nu^{\ve}_{u,a} \otimes\delta_{\lambda^\ve} \}_{0<\ve \le 1}$
satisfies an LDP on $G\Omega^H_{\alpha} ( {\mathbb R}^d) \times C_0^{1-H}([0,1], {\mathbb R})$
with the good rate function which is
defined for $({\bf w}, \lambda)$ as follows: 
\[
\hat{I}_1 ({\bf w})
+\infty \cdot {\bf 1}_{\{0\}^c} (\lambda)
=
\begin{cases}
  \tfrac12  \|h\|^2_{\cH} & (\mbox{if ${\bf w}= {\cal L}(h)$ for some $h \in {\cal Q}^{u,a}$ and $\lambda =0$}), \\
    \infty &  (\mbox{otherwise}).
  \end{cases}
  \]

Choose an embedding as in Remark \ref{rem.embed}
and consider RDE \eqref{rde_x.def} with $V_i = A_i$, 
$x=u$ and $n =M$. 
Define a continuous map 
$\Psi\colon G\Omega^H_{\alpha} ( {\mathbb R}^d) \times C_0^{1-H}([0,1], {\mathbb R}) 
\to C_{x} ([0,1], \cM)$ by 
$\Psi ({\bf w}, \lambda)_t = \pi ( x + \Phi({\bf w}, \lambda)^1_{0,t} )$.
Then, 
$\Psi(\ve {\bf W}, \lambda^\ve)$ is 
an $\infty$-quasi-continuous modification of $X^{\ve} = \pi (U^{\ve})$,
where $U^{\ve}$ solves SDE \eqref{eq.sde_on_P}.
In what follows we use this version of $X^{\ve}$.
Note that $\psi (h) = \Psi (\cL (h), 0)$ for $h \in \cH$.

Now we claim that the law of $\Psi$ under 
$\hat\nu^{\ve}_{u,a} \otimes \delta_{\lambda^\ve}$
is the pinned
$\ve^2 (\Delta_{\mathrm{sub}}/2 +V)$-diffusion measure $\Q^{\ve}_{x,a}$
from $x=\pi (u)$ to $a$.
Let $k \ge 1$, $G \in C^\infty (\cM^k)$ 
and $0 =t_0 <t_1 <\cdots < t_k <t_{k+1}=1$ be arbitrary.
Then, we have
\begin{align*}
\lefteqn{
\int  G (\Psi ({\bf w}, \lambda)_{t_1}, \ldots, \Psi ({\bf w}, \lambda)_{t_k}) 
\hat\nu^{\ve}_{u,a} \otimes \delta_{\lambda^\ve}
(d{\bf w} d\lambda)
}
\\
&= 
\int  G (\Psi (\ve {\bf W}, \lambda^\ve)_{t_1}, \ldots, \Psi (\ve {\bf W}, \lambda^\ve)_{t_k}) \hat\nu^{\ve}_{u,a} (dw)
\\
&= 
p_1^\ve (x,a)^{-1}
{\mathbb E} [G (X^\ve_{t_1}, \ldots, X^\ve_{t_k})  \delta_a (X^{\ve}_1)]
\\
&= 
p_1^\ve (x,a)^{-1}
\int_{\cM^k} G (x_1, \ldots, x_k) \prod_{i=0}^k p^\ve_{t_{i+1} -t_i} (x_i, x_{i+1})
\prod_{i=1}^k {\rm vol} (dx_i)
\end{align*}
as desired. 
Here, we set $x_0 =x$ and $x_{k+1} =a$ for simplicity.
This proves our claim.
Note that this argument also proves the existence of 
the pinned diffusion measure $\Q^{\ve}_{x,a}$.

By the above fact and 
Lyons' continuity theorem, we can use the contraction principle
(\cite[Theorem 4.2.1]{dzbk}) to prove that $\{\Q^{\ve}_{x,a} \}_{0 <\ve \le 1}$ satisfies an LDP with a good rate function.
Since $\psi$ is a bijection that preserves the energy 
(Proposition \ref{pr.iso_path}), 
the rate function $J$ is given by \eqref{eq.0720-2}.
This completes the proof of Theorem \ref{thm.mainQ}.
\end{proof}

\section{Lower estimate}\label{sec.lower}

In this section we prove the lower estimate of LDP in 
Theorem \ref{tm.ldp.dlt} {\rm (i)}.
Take any $u \in \cP$ and $a \in \cM$. 
For this $u$,  $\phi (h)$ denotes the unique solution of
ODE \eqref{eq.odeFB2} and $\psi (h)$ is its projection on $\cM$.
The subset ${\cal Q}^{u, a} \subset \cH$ is defined in \eqref{def.Qua}.

\begin{lemma}\label{pr.CMnice}
For every $h \in {\cal Q}^{u, a}$
there exists a 
sequence $\{ h_j\}_{j=1}^\infty$ in ${\cal Q}^{u, a}$
which satisfies the following conditions:
\begin{itemize}
\item[{\rm (i)}] $\lim_{j\to\infty} \|h_j -h\|_{\cH} =0$.
\item[{\rm (ii)}] 
For all $j$, $D\psi (h_j)_1\colon \cH \to T_a\cM$ is surjective, where 
$D\psi (h_j)_1$ denotes the tangent map of 
$\cH \ni k \mapsto \psi (k)_1 \in \cM$ at $h_j$.
\item[{\rm (iii)}]
For all $j$,  $\langle h_j, \bullet\rangle_{\cH}\in \cW^*$, that is, 
this linear functional on $\cH$ extends 
to a bounded linear functional on $\cW$.
\end{itemize}
\end{lemma}

\begin{proof}
First we will find $\{ h_j\}_{j=1}^\infty$ which satisfies (i) and (ii) only.
Let $\{Z_1, \ldots, Z_d\}$ be an orthonormal frame of $\cD$
on a coordinate neighborhood $U$ of $x =\pi (u)$.
We view $U$ as an open subset of $\R^n$ and extend
$Z_i$, $1\le i \le d$, as a smooth vector field on $\R^n$
with compact support.
We denote by $\cH_\tau$, $\tau \in (0,1)$,
the Hilbert space of $\R^d$-valued Cameron-Martin paths 
defined on the time interval $[0, \tau]$.

Consider the following ODE on $\R^n$
driven by a Cameron-Martin path $k$:
\[
d x(k)_t 
= \sum_{i=1}^n Z_i (x(k)_t ) d k_t,
\qquad x(k)_0 =x.
\]
The following fact was proved in \cite[Section 3]{in2}. 
For every $j$ large enough, there exists 
a Cameron-Martin path $k_j\in \cH_{1/j}$ 
such that $|k^{\prime}_{j,t} | \le 1$ for almost all $t \in [0, 1/j]$ and
$D x(k_j)_{1/j}$ is surjective.
(Precisely,  
$\{Z_i\}$ is assumed in to satisfy
 the bracket generating condition 
at every point of $\R^n$ in \cite{in2}, while the condition 
is assumed  only on $U$ here.
However, this difference does not matter at all 
since when $j$ is large enough, the Cameron-Martin norm of 
$x(k_j)$ is small enough and therefore $x(k_j)$ stays inside $U$, anyway.)

Then, since $\{Z_i\}$ are orthonormal,
$|x (k_j)^{\prime}_t | \le 1$ for almost all $t \in [0, 1/j]$.
Denote by $l_j:=\psi^{-1} (x (k_j)) \in \cH_{1/j}$ 
the anti-development of $x (k_j)$,
then $|l^{\prime}_{j,t} | \le 1$ for almost all $t \in [0, 1/j]$, too.
Define $h_j \in \cH$ by 
\[
h_{j,t} :=
\begin{cases}
l_{j, t} & \mbox{on $t \in [0, 1/j]$}, \\
 l_{j, (2/j) -t}   & \mbox{on $t \in [1/j, 2/j]$}, \\
    h_{\tau}  \qquad \mbox{with $\tau = \frac{t - (2/j)}{1- (2/j)}$} &  \mbox{on $t \in [2/j, 1]$}.
  \end{cases}
  \]
Since ODE \eqref{eq.odeFB2} has no drift term, 
we can easily see that $\phi (h_j)_{2/j} =u$ and $\phi (h_j)_1 =\phi (h)_1$. In particular, $h_j \in {\cal Q}^{u,a}$.
It is a routine to check {\rm (i)} (see \cite{in2} for a proof for instance).

Next we show  {\rm (ii)}. Fix $j$ and
consider the admissible path $\psi (h_j)\colon [0,1] \to \cM$.
We can find a partition of $0=s_0 < s_1 <\cdots <s_N =1$
of $[0,1]$ such that, for all $1\le r\le N$,
 $\psi (h_j)\restriction_{[s_{r-1}, s_r]}$
is contained in a  coordinate neighborhood $U_r$ on which
an orthonormal frame $\{Z^{(r)}_1, \ldots, Z^{(r)}_d\}$ 
of $\cD$ can be chosen.
We may assume $s_1 =1/j$,
$U_1 =U$ and $Z^{(1)}_i =Z_i$ for all $1\le i \le d$.

Obviously,
there exists a unique Cameron-Martin path 
$k^{(r)}\colon [s_{r-1}, s_r] \to \R^d$ such that
 $\psi (h_j)\restriction_{[s_{r-1}, s_r]}$ satisfies the following 
 ODE with the initial condition $x^{(r)}_{s_{r-1}}=\psi (h_j)_{s_{r-1}}$:
\begin{equation}\label{eq.ode_End}
d x^{(r)}_t
= \sum_{i=1}^n Z^{(r)}_i (x^{(r)}_t ) d k^{(r)}_t
\qquad  \mbox{on $[s_{r-1}, s_r]$}.
\end{equation}
When the initial value of ODE \eqref{eq.ode_End}
is replaced by $\xi^{(r-1)}$,
we write ${\rm End}^{(r)} (\xi^{(r-1)} ):= x^{(r)}_{s_r}$.
When $\xi^{(r-1)}$ is close enough to $\psi (h_j)_{s_{r-1}}$, 
${\rm End}^{(r)} (\xi^{(r-1)} )$ is well-defined.
By the theory of flow of diffeomorphisms for ODEs, 
${\rm End}^{(r)}$ is a local diffeomorphism 
from a neighborhood of $\psi (h_j)_{s_{r-1}}$ to  
a neighborhood of $\psi (h_j)_{s_r}$.
(Note that here and in what follows, $k^{(r)}$ is fixed.)
Hence, 
${\rm End} := {\rm End}^{(N)} \circ \cdots \circ {\rm End}^{(2)}$
is a local diffeomorphism 
from a neighborhood of $\psi (h_j)_{s_1}$ to  
a neighborhood of $\psi (h_j)_{1}=a$.
In particular, the tangent map of ${\rm End}$ is
bijective from $T_b\cM$ to $T_a \cM$,
where we write $b:=x(k_j)_{s_1} = \psi (h_j)_{s_1}$.

Now consider $\psi (h_j)\restriction_{[0, s_1]} =x (k_j)$.
Since $Dx(k_j)_{s_1}$ is surjective, for every 
$v\in T_b\cM$
there exist sufficiently small $\ve_0 >0$ and 
a $C^1$-curve $(-\ve_0,\ve_0)\ni \ve \mapsto k_j^{\ve}
\in \cH_{1/j}$ such that
$k_j^{0} =k_j$ and $(d/d\ve)|_{\ve =0} x(k^{\ve}_j)_{s_1} =v$.
(Below $\ve_0$ may change from line to line.)

Take $u \in T_a \cM$ arbitrarily and let $v$ 
be the unique element of $T_b\cM$ which corresponds to $u$
through the bijection.
For this $v$, we take $k_j^{\ve}$ as above.
Define $h_j^{\ve}$ to be the unique element in $\cH$ such that
(1) $\psi (h_j^{\ve})$ coincides with $x (k^{\ve}_j)$ on $[0, s_1]$
and (2)  $\psi (h^{\ve}_j)\restriction_{[s_{r-1}, s_r]}$
solves ODE \eqref{eq.ode_End} for all $2 \le r \le N$.
By way of construction, $h_j^{0}=h_j$
and $(d/d\ve)|_{\ve =0}  \psi (h^{\ve}_j)_1 =u$.
Thus, we have shown {\rm (ii)}.

Finally, we cope with {\rm (iii)}.
Since $h \mapsto \psi (h)_1$ is Fr\'echet-$C^1$
and the inclusion $\cW^* \hookrightarrow \cH^* \cong \cH$
is continuous and dense, we may use 
a topological lemma \cite[Lemma 7.3]{in1}.
It implies that we can find 
$\hat{h}_j\in {\cal Q}^{u, a}$ wich satisfies the following condition
for sufficiently large $j$:
(1) $\| h_j- \hat{h}_j \|_{\cH} \le 1/j$,  
(2) $D\psi (\hat{h}_j)_1\colon \cH \to T_a\cM$ is surjective,
and (3) $\langle \hat{h}_j, \bullet\rangle_{\cH}\in \cW^*$.
Hence, $\{ \hat{h}_j\}_{j=1}^\infty$ is the desired sequence.
(Precisely, \cite[Lemma 7.3]{in1} is for 
$\R^n$-valued Fr\'echet-$C^1$ maps. However, it still holds 
for our manifold-valued case without modification
since information outside
a sufficiently small neighborhood of $a$ is not used in its proof.)
\end{proof}

For $R>0$, we set 
\begin{equation}\label{def.hatball}
\hat{B}_{R}=\{ {\bf w} \in G\Omega^B_{  \al, 4m} ({\mathbb R}^{d}) 
\mid
\|{\bf w}^1 \|_{ \al, 4m -B}^{4m}
 +  \|{\bf w}^2 \|_{2 \al, 2m -B}^{2m} < R^{4m}
 \}
\end{equation}
and set
$\hat{B}_{R} ({\bf h})= \tau_h ( \hat{B}_{R} )$, 
where $\tau_h$ is the Young translation by $h\in \cH$ on $G\Omega^B_{\al, 4m} ({\mathbb R}^{d})$.
Since $\tau_h$ is a homeomorphism from $G\Omega^B_{\al, 4m} ({\mathbb R}^{d})$ to itself,
$\{ \hat{B}_{R} ({\bf h}) \mid R>0\}$ forms a fundamental system of open neighborhoods around ${\bf h}=\cL (h)$.

\begin{proposition}\label{pr.low_ball}
Assume that $h \in {\cal Q}^{u, a}$ satisfies that
$D\psi (h)_1\colon \cH \to T_a\cM$ is surjective
and that $\langle h, \bullet\rangle_{\cH}\in \cW^*$.
Then, there exists a constant $c =c(h)>0$ 
independent of $R$ such that 
\begin{equation}\label{eq.0604}
\liminf_{\ve \searrow 0} \ve^2
\log \nu^{\ve}_{u, a}( \hat{B}_R ({\bf h}) )
\ge  
- \frac12 \|h\|_{\cH}^2 -cR
\end{equation}
holds for every sufficiently small $R>0$.
\end{proposition}

\begin{proof}
In this proof $R \in (0, R_0)$ and $\ve \in (0,\ve_0)$, where
$R_0 >0$ and $\ve_0 >0$ are sufficiently small constants
and may vary from line to line.

Take a coordinate neighborhood $\tilde{U}$ of $a$.
We also view $\tilde{U}$ as a bounded open subset of $\R^n$.
By applying a dilation on $\R^n$ to $\tilde{U}$ if necessary
we may also assume that $\delta_a$ on $\cM$ (with respect to $\mathbf{vol}$) corresponds to the usual $\delta_a$ on $\R^n$
(with respect to the standard Lebesgue measure).
We also take another open subset $U$ so that
$a \in U \subset \bar{U}  \subset \tilde{U}$.

For $F \in \mathbf{D}_{2,1} (\cM)$  that take values in 
 $\tilde{U}$ (or $F \in \mathbf{D}_{2,1} (\cM)$ restricted 
 to a subset of $\{ F \in  \tilde{U}\}$), 
 there are two Malliavin covariance matrices.
One is the original one defined with respect to
the Riemannian metric $\cM$,
while the other is the standard one for $\R^n$-valued
Wiener functionals via the inclusion $\tilde{U}\subset \R^n$.
Since there is a constant $C=C(\tilde{U})>0$ 
such that
\begin{equation}\label{eq.MC_hikaku}
C^{-1} \det \sigma_F (w) \le
\det \tilde\sigma_F (w) \le C  \det \sigma_ F(w)\qquad
\mbox{on $\{ w\in \cW \mid F (w) \in  \tilde{U}\}$},
\end{equation}
either one of the two works.
In what follows we will use the standard one for $\R^n$-valued
Wiener functionals and denote it by $\sigma_F$ again
by slightly abusing the notation.

For $b \in \pi^{-1} (a)$,
there exist a smooth map $\hat\pi \colon \cP \to \R^n$
and a open neighborhood $V$ of $b$ such that
$V \subset \pi^{-1} (\tilde{U})$ and
 $\pi\restriction_V \equiv \hat\pi\restriction_V$.
We can extend $\hat\pi$ again  so that it is 
 a smooth function $\hat\pi \colon \R^M \to \R^n$ with 
 compact support. 
 (Recall the embedding $\iota\colon \cP \hookrightarrow \R^M$ in Remark \ref{rem.embed}.)
We write $\hat{X}_t^{\ve} = \hat\pi (U_t^{\ve} )$.  
When we need to specify the dependency on $u$ and $w$,
we write $\hat{X}_t^{\ve} = \hat{X}^\ve (t,u,w)$.
(We will use the same notation for $X_t^{\ve}$ and $U_t^{\ve}$, too.)
  
Let $\chi:{\mathbb R} \to {\mathbb R}$ be as in Proposition \ref{pr.comp2}.
Moreover, we assume that $\chi$ is even and
non-increasing on $[0,\infty)$
so that $\chi$ takes values in $[0,1]$.
Take any $f \in C^\infty (\R^n, [0, \infty))$ whose support is 
contained in the unit ball and set $f_j  = j^n f(j\,\cdot\,)$.
Then, $\lim_{j\to \infty} f_j =\delta_0$ in $\cS^\prime (\R^n)$.
If we set $f_j^{\ve, a}  = \ve^{-n} f_j ((\,\cdot\, -a)/\ve)$, 
then $\lim_{j\to \infty} f_j^{\ve, a} =\delta_a$ in $\cS^\prime (\R^n)$.
There exists $j_0 >0$ such that the support of 
$f_j^{\ve, a}$ is contained in $U$ for every $j \ge j_0$ and $\ve \in (0,1]$.
In that case, $f_j^{\ve, a}$ can also be viewed as a function on $\cM$.
We will assume $j \ge j_0$ and set $b:= \phi (h)_1 \in \cP$.

Then, it holds that
\begin{align}
\nu^{\ve}_{u, a}( \hat{B}_R ({\bf h}) ) 
&=
\int  I_{\hat{B}_R ({\bf h})} ({\bf w})  \nu^{\ve}_{u, a} (d {\bf w})
=
\int  I_{\hat{B}_R } (\tau_{-h} ({\bf w}))  \nu^{\ve}_{u, a} (d {\bf w})
\nn
\\
&=
\int  I_{\hat{B}_R } (\tau_{-h} (\ve {\bf W}))  \theta^{\ve}_{u, a} (dw)
\nn
\\
&\ge
\int  \chi \left( 
\frac{ \|\tau_{-h} (\ve {\bf W})^1 \|_{\alpha, 4m -B}^{4m} 
+ 
\|\tau_{-h} (\ve {\bf W})^2 \|_{2\alpha, 2m -B}^{2m}}{R^{4m}}
 \right)  \theta^{\ve}_{u, a} (dw)
\nn
\\
&= 
{\mathbb E}  \Bigl[   
\chi \left( 
\frac{ \|\tau_{-h} (\ve {\bf W})^1 \|_{\alpha, 4m -B}^{4m} 
+ 
\|\tau_{-h} (\ve {\bf W})^2 \|_{2\alpha, 2m -B}^{2m}}{R^{4m}}
 \right) 
 \delta_{a} (X^{\ve}_1) 
 \Bigr]
 \nn\\
 &= 
 \lim_{j\to\infty}
 {\mathbb E}  \left[   
\chi \left( 
\frac{ \|\tau_{-h} (\ve {\bf W})^1 \|_{\alpha, 4m -B}^{4m} 
+ 
\|\tau_{-h} (\ve {\bf W})^2 \|_{2\alpha, 2m -B}^{2m}}{R^{4m}}
 \right) 
 f_j^{\ve, a} (X^{\ve}_1) 
 \right]
 \nn\\
 &= 
 \lim_{j\to\infty}
 {\mathbb E}  \left[   
\chi \left( 
\frac{ \|\tau_{-h} (\ve {\bf W})^1 \|_{\alpha, 4m -B}^{4m} 
+ 
\|\tau_{-h} (\ve {\bf W})^2 \|_{2\alpha, 2m -B}^{2m}}{R^{4m}}
 \right) 
 f_j^{\ve, a} (\hat{X}^{\ve}_1) 
 \right]
\nn\\
&=
e^{ - \| h\|^2_{{\cal H}}/2\ve^2 }  
\lim_{j\to\infty} {\mathbb E} \Bigl[ e^{ - \la h, w\ra /\ve } 
  \chi \left( 
\frac{ \| (\ve {\bf W})^1 \|_{\alpha, 4m -B}^{4m} 
+ 
\| (\ve {\bf W})^2 \|_{2\alpha, 2m -B}^{2m}}{R^{4m}}
 \right) 
 \nn
 \\
 & \qquad\qquad \quad \qquad\qquad \quad
  \times \ve^{-n} 
  f_j \left(  \ve^{-1}[\hat{X}^{\ve}(1, u, w+ \frac{h}{\ve}) -a ]\right) 
  \Bigr]
  \nn\\
&\ge
e^{ - \| h\|^2_{{\cal H}}/2\ve^2 }  e^{-cR/\ve^2}
\ve^{-n} 
\lim_{j\to\infty} {\mathbb E} \Bigl[ 
  \chi \left( 
\frac{ \| (\ve {\bf W})^1 \|_{\alpha, 4m -B}^{4m} 
+ 
\| (\ve {\bf W})^2 \|_{2\alpha, 2m -B}^{2m}}{R^{4m}}
 \right) 
 \nn
 \\
 & \qquad\qquad \quad \qquad\qquad \qquad\quad
  \times f_j \left(  \ve^{-1}[
\hat{X}^{\ve}(1, u, w+ \frac{h}{\ve}) -a ]\right) \Bigr].  
\label{eq.210609-1}
\end{align}
Note that when $R>0$ and $\ve >0$ are sufficiently small, 
$U_1^{\ve}$ is close enough to $b$ and therefore we have
$X_1^{\ve}= \hat{X}_1^{\ve}$.
We used Cameron-Martin formula, too.
We now check the last inequality.   
Form the assumption that $\langle h, \bullet\rangle_{\cH}\in \cW^*$
 and the fact that $(\alpha, 4m)$-Besov norm 
(of the first level path)  is stronger  than the usual 
sup-norm, we have $|\langle h, w\rangle_{\cH} |\le c R/\ve$
for a certain constant $c>0$ if 
$\| (\ve {\bf W})^1 \|_{\alpha, 4m -B} \le R$.

Set 
\begin{align*}
\xi_{\ve} &= R^{-4m}  (\| (\ve {\bf W})^1 \|_{\alpha, 4m -B}^{4m} 
+ 
\| (\ve {\bf W})^2 \|_{2\alpha, 2m -B}^{2m} ),
\\
F_{\ve} &=\ve^{-1}\left[\hat{X}^{\ve}\left(1, x, w+ \frac{h}{\ve} \right) -a \right].
\end{align*}
Let $\Gamma$ as in \eqref{def.Gamma}
(with $(x_t) =(U_t)$, $G =\hat\pi$ and $m =M$).
Then, Malliavin covariance of 
 $F_{\ve}$  
 equals $\Gamma (\tau_h (\ve {\bf W}) ,\lambda^{\ve})_1$,
which tends to 
$\Gamma ({\bf h} , 0)_1$ as $(\ve {\bf W}, \lambda) \to ({\bf 0}, 0)$.
Note that $\Gamma ({\bf h} , 0)_1$ is the deterministic 
Malliavin covariance of $\psi_1$ at $h$. 
As is well-known, $D\psi (h)_1$ is surjective 
if and only if $\Gamma ({\bf h} , 0)_1$ is a strictly positive 
symmetric matrix.
Hence, 
when $R_0>0$ and $\ve_0 >0$ are sufficiently small, 
there exists a constant $\rho >0$ such that 
Condition \eqref{wat1.eq} holds 
for every $\ve \in(0,\ve_0]$ and $R\in (0,R_0]$.
($\rho$ does not depend on $\ve$ or $R$.)
Now we can apply Proposition \ref{pr.comp1} to 
the right hand side of \eqref{eq.210609-1} to obtain 
\[
\nu^{\ve}_{x, a}( \hat{B}_R ({\bf h}) ) 
\ge
e^{ - \| h\|^2_{{\cal H}}/2\ve^2 }  e^{-cR/\ve^2}\ve^{-n} 
{\mathbb E}[\chi (\xi_{\ve}) \delta_0 (F_{\ve})].
\] 
It suffices to prove that 
$\lim_{\ve\searrow 0} {\mathbb E}[\chi (\xi_{\ve}) \delta_0 (F_{\ve})]$
exists and strictly positive.
Since $\chi$ is constant near the origin, 
 \eqref{wat3.eq} clearly holds with $\xi_0 =0$.
We will check \eqref{wat2.eq}. 
 
Since
$(U^{\ve}_t)$ is viewed as the solution of the vector space-valued
SDE with small noise, its asymptotic behavior as $\ve\searrow 0$
is well-known:
\[
U^{\ve}  \left(1, u, w+ \frac{h}{\ve} \right) 
 = b+  \ve \eta_1 + O (\ve^2) 
 \qquad \mbox{in ${\bf D}_{\infty} ({\mathbb R}^M)$},
 \]
where $\eta_1 =\eta_1 (w)$ is a certain element of
 the first order Wiener chaos (which can actually be written down 
explicitly as a Wiener integral). 
It is easy to see from this that  
\[
F_{\ve} = \nabla \hat\pi (b) \langle \eta_1\rangle
+ O (\ve) 
 \qquad \mbox{in ${\bf D}_{\infty} ({\mathbb R}^n)$}.
 \]
$\nabla \hat\pi (b) \langle \eta_1\rangle$  is also 
an element of  the first order Wiener chaos and therefore
induces a Gaussian measure of mean zero on ${\mathbb R}^n$.
Since its covariance matrix equals $\Gamma ({\bf h} , 0)_1$,
the Gaussian measure is non-degenerate 
and hence its probability density function is strictly positive 
at the origin. 
From this and  Proposition \ref{pr.comp2}  we see that
\[
\lim_{\ve \searrow 0}
 {\mathbb E} [ \chi(\xi_{\ve}) \delta_0  (F_{\ve})] = 
 {\mathbb E} [\delta_0 (\nabla \hat\pi (b) \langle \eta_1\rangle )]
 \in (0,\infty).
 \]
This completes the proof of Proposition \ref{pr.low_ball}. 
 \end{proof}

\begin{proof}[Proof of the lower estimate in Theorem \ref{tm.ldp.dlt} {\rm (i)}.]
Let $O \subset G\Omega^B_{\al, 4m} ({\mathbb R}^d)$ be an open 
set with
$\inf_{{\bf w} \in O} I_1 ({\bf w})  <\infty$.
Then, we see from Lemma \ref{pr.CMnice} that
for every $\kappa > 0$ we can find $h \in {\cal Q}^{u, a}$ such that
(1) ${\bf h} \in O$ and 
(2) $0\le ( \|h\|^2_{\cH}/2 )- \inf_{{\bf w} \in O} I_1 ({\bf w}) <\kappa$
and (3) $h$ satisfies the assumption of Propositions \ref{pr.low_ball}.
Using Proposition \eqref{pr.low_ball}, we have 
\[
\liminf_{\ve \searrow 0} \ve^2
\log \nu^{\ve}_{u, a}(O) \ge
\liminf_{\ve \searrow 0} \ve^2
\log \nu^{\ve}_{u, a}( \hat{B}_R ({\bf h}) )
\ge  
-  \inf_{{\bf w} \in O} I_1 ({\bf w})     -cR -\kappa
\]
for every sufficiently small $R>0$.
Letting $R\searrow 0$ and then $\kappa\searrow 0$, we have the 
desired lower estimate.  
\end{proof}

\section{Upper estimate} \label{sec.upper}

In this section we prove the upper estimate of LDP in 
Theorem \ref{tm.ldp.dlt} {\rm (i)}.
By a standard argument in the large deviation theory, 
we can easily prove it by using
Proposition \ref{pr.exp_tight} and 
Proposition \ref{pr.limsup_ball}, which will be given below.
(That $\nu^{\ve}_{u,a}$ may not be a probability measure is 
irrelevant in this part.)
In this section the positive constant $\nu$ varies from line to line.

The key of proving  the upper estimate is a localized version 
of IbP formula.
This type of IbP formula is not new.
For example, it was used in the proofs of 
Propositions \ref{pr.comp1} and  \ref{pr.comp2}.
In the manifold-valued Mallaivin calculus in \cite{ta},  
a quite similar argument already appeared.
Concerning this, see also \cite{yo}.

We consider the projected process 
 $X^{\ve}= \pi (U^{\ve})$  defined by SDE \eqref{eq.sde_on_P}.
Let  $a \in U \subset \bar{U}  \subset \tilde{U}$
as in \eqref{eq.MC_hikaku}.
By taking $\tilde{U}\subset \R^n$ slightly smaller if necessary, 
we can extend
the coordinate functions on $\tilde{U} \ni x \mapsto x^i \in \R$, $1 \le i \le n$,
 to a smooth function on $\cM$, which is denoted by $\beta^i$.
If we set $\hat{X}^{\ve}_1 
=\{\beta^i(X^{\ve}_1)\}_{i=1}^n$, 
then $\hat{X}^{\ve}_1 \in \mathbf{D}_{\infty} (\R^n)$
and $\hat{X}^{\ve}_1= X^{\ve}_1$ on $\{w\mid X^{\ve}_1(w) \in \tilde{U} \}$.
It should be recalled that the Sobolev norm
$\| \hat{X}^{\ve}_1\|_{p,k}$ is bounded in $\ve \in (0,1]$ for every 
$1<p<\infty$ and $k \in \N$.

\begin{lemma}\label{lem.0624}
Let the notation be as above 
and write $F_{\ve} :=\hat{X}^{\ve}_1$ for notational simplicity.
Suppose that $\eta_1, \eta_2 \colon \R^n \to [0,1]$ be smooth functions 
with compact support in $\tilde{U}$ such that $\eta_2 \equiv 1$ 
on the support of $\eta_1$.
Then, for every $f \in \cS (\R^n)$, $G \in \mathbf{D}_{p,k}$
($1<p<\infty, k\in \N_+$) and $1\le i \le n$,
the following assertions hold true:
\\
{\rm (i)}~$\gamma_{F_{\ve}}   \eta_1 (F_{\ve})
 \in \mathbf{D}_\infty ({\rm Mat} (n,n))$. 
Here, $\gamma_{F_{\ve}}$ is the inverse of the
Malliavin covariance matrix $\sigma_{F_{\ve}}$.
 \\
{\rm (ii)}~If we set 
\[
\Phi_i^{\eta_1} (\,\cdot\, ;G) = \sum_{j=1}^d D^* 
\left( \gamma^{ij}_{F_{\ve}}  \eta_1 (F_{\ve}) \cdot 
G \cdot DF_{\ve}^{j}  \right),
\]
then $\Phi_i^{\eta_1} (\,\cdot\, ;G)\in \mathbf{D}_{p',k-1}$ for 
every $p' \in (1, p)$.
Moreover, there exist positive 
 constant $c=c_{p,p'}$ and $\nu$ 
such that 
\begin{equation}\label{eq.0625-1}
\| \Phi_i^{\eta_1} (\,\cdot\, ;G) \|_{p', k-1} = c \ve^{-\nu} \|G\|_{p,k},
\qquad
0<\ve \le 1.
\end{equation}
Here, $c=c_{p,p'}$ and $\nu$ are independent of $\ve$ and $G$.
Moreover, $\nu$ does not depend on $(p,p')$, either.
\\
{\rm (iii)}~We have
\begin{align}\label{eq.loc.ibp}
{\mathbb E}[\partial_i f (F_{\ve}) \eta_1 (F_{\ve}) G]
&=
{\mathbb E}[ f (F_{\ve})  \eta_2 (F_{\ve})
\Phi_i^{\eta_1}(\,\cdot\,;G) ].
\end{align}
\end{lemma}

\begin{proof}
By \eqref{eq.KS_est} and  \eqref{eq.MC_hikaku}, 
there exist positive constants $C_p, \nu$ such that 
\begin{equation}\label{eq.loc_KS_est}
{\mathbb E}[
(\det \sigma_{F_\ve })^{-p}  {\bf 1}_{\{F_{\ve} \in \tilde{U} \}}
]^{1/p}  \le C_p \ve^{-\nu},
\qquad
0 < \ve \le 1.
\end{equation}
Here, $C_p$ and $\nu$ are independent of $\ve$.
Moreover, $\nu$ does not depend on $p$, either.

Now we prove {\rm (i)}.
It may be heuristically obvious, but we must take care of the possibility
 that $\gamma_{F_{\ve}}$
is not defined outside $\{F_{\ve} \in \tilde{U} \}$.
For $m \in \N_+$ we set 
$\sigma_{F_\ve}^m :=\sigma_{F_\ve} + m^{-1} {\rm Id}_n$,
where ${\rm Id}_n$ stands for the identity matrix of size $n$.
Then, $(\det \sigma^m_{F_\ve })^{-1} \le m^n$, a.s. 
and its inverse
$\gamma^m_{F_{\ve} }:=(\sigma^m_{F_\ve })^{-1}$ exists.
Moreover, 
$(\det \sigma^m_{F_\ve })^{-1} 
\nearrow (\det \sigma_{F_\ve })^{-1} \in [0,\infty]$, a.s.
Note that $\gamma^m_{F_{\ve} }$ and 
$(\det \sigma^m_{F_\ve })^{-1}$ are both ${\mathbf D}_\infty$-functionals defined on $\cW$.
Recall that 
\[
\gamma^m_{F_{\ve} }
=
(\det \sigma^m_{F_\ve })^{-1}\times
[\mbox{the adjugate matrix of $\sigma^m_{F_\ve }$}].
\]
From this and \eqref{eq.loc_KS_est}, we can easily see that 
$(\gamma^m_{F_{\ve}})^{ij}   \eta_1 (F_{\ve})\to \gamma^{ij}_{F_{\ve}}  \eta_1 (F_{\ve})$ in $L^p$ as $m\to\infty$ ($1<p <\infty$).

Now we calculate the first order derivative.
Note that 
\begin{align}
D\{ (\gamma^m_{F_{\ve}})^{ij}  \eta_1 (F_{\ve})\}
&= -\sum_{k,l} (\gamma^m_{F_{\ve}})^{ik}  \cdot
D\sigma_{F_{\ve}}^{kl} \cdot
(\gamma^m_{F_{\ve}})^{lj}  \cdot
\eta_1 (F_{\ve})
\nn\\
&\qquad +
\sum_{l}
(\gamma^m_{F_{\ve}})^{ij}   \cdot 
\partial_l \eta_1 (F_{\ve})  \cdot DF_{\ve}^l.
\nn
\end{align}
By the same reason as above, this belongs to ${\mathbf D}_\infty (\cH)$ and converges in $L^p (\cH)$, $1<p <\infty$, as $m\to\infty$.
The closability of $D$ implies that
\begin{align}
D\{ \gamma^{ij}_{F_{\ve}}  \eta_1 (F_{\ve})\}
&= -\sum_{k,l} \gamma^{ik}_{F_{\ve}} 
\cdot
D\sigma_{F_{\ve}}^{kl} \cdot
\gamma_{F_{\ve}}^{lj} \cdot \eta_1 (F_{\ve})
+
\sum_{l} \gamma^{ij}_{F_{\ve}}   \cdot 
\partial_l \eta_1 (F_{\ve})  \cdot DF_{\ve}^l
\label{eq.0624-1}
\end{align}
and $\|\gamma^{ij}_{F_{\ve}}  \eta_1 (F_{\ve}) \|_{p,1} = O( \ve^{-\nu})$ for every $1<p<\infty$.
Repeating essentially the same argument for higher order derivatives, 
we can prove that $\|\gamma^{ij}_{F_{\ve}}  \eta_1 (F_{\ve}) \|_{p,k} \le O( \ve^{-\nu})$ for every $1<p<\infty$ and $k \in \N_+$
(if we adjust the value of $\nu >0$).
Thus, we have shown {\rm (i)}.

To prove {\rm (ii)}, just recall that $D^*$ is a bounded linear 
map from $\mathbf{D}_{p,k} (\cH)$ to $\mathbf{D}_{p,k-1}$
for every $p$ and $k$ (see \cite[p. 365]{iwbk} for instance).

Finally, we prove {\rm (iii)}.
From a well-known formula for $D^*$, we see that
\[
\Phi_i^{\eta_1} (\,\cdot\, ;G) 
= -\sum_{j=1}^d \{
\left\la D(\gamma^{ij}_{F_{\ve}} \cdot  \eta_1 (F_{\ve}) \cdot 
G),  DF_{\ve}^j  \right\ra_\cH
+
\gamma^{ij}_{F_{\ve}} \cdot  \eta_1 (F_{\ve}) \cdot 
G \cdot LF_{\ve}^j
\},
\]
where $L =-D^*D$ is the Ornstein-Uhlenbeck operator.
Since this vanishes outside $\{ F_\ve \in {\rm supp} (\eta_1)\}$,
we  have $\Phi_i^{\eta_1} (\,\cdot\, ;G)  = \eta_2 (F_{\ve})\Phi_i^{\eta_1} (\,\cdot\, ;G)$.
As in the proof for 
the standard IbP formula in \eqref{ipb2.eq}, 
we see from the definition of $D^*$ that
\begin{align*}
{\mathbb E}[\partial_i f (F_{\ve}) \eta_1 (F_{\ve}) G]
&=
{\mathbb E}[\la D f (F_{\ve}),  
\sum_{j=1}^d\gamma^{ij}_{F_{\ve}} DF_{\ve}^j  \ra_{\cH} 
\,   \eta_1 (F_{\ve}) G]
=
{\mathbb E}[ f (F_{\ve}) \Phi_i^{\eta_1} (\,\cdot\, ;G)].
\end{align*}
This completes the proof of Lemma \ref{lem.0624}.
\end{proof}

Take a smooth function 
$\eta_j \colon \R^n \to [0,1]$, $1 \le j \le 2n+1$, 
with compact support in $\tilde{U}$ with the following properties:
(1) $\eta_1 \equiv 1$ on $U$,
(2) $\eta_{j+1} \equiv 1$ on the support of $\eta_j$ for all $1 \le j \le 2n$.
We write $\eta := \{\eta_i\}_{i=1}^{2n+1}$.

Let $\beta :=(i_1, i_2, \ldots, i_s)$ be a multi-index of length at most $2n$,
that is, $s \le 2n$ and $1 \le i_1, \ldots, i_s \le n$.
For $\beta =(i_1)$,  just set  
$\Phi_{\beta}^{\eta} (\,\cdot\, ;G)=\Phi_{i_1}^{\eta_1} (\,\cdot\, ;G)$.
For $\beta =(i_1, i_2)$,  set
$\Phi_{\beta}^{\eta} (\,\cdot\, ;G)
= \Phi_{i_2}^{\eta_2} (\,\cdot\, , \Phi_{i_1}^{\eta_1} (\,\cdot\, ;G))$.
When $s \ge 3$, write $\beta' = (i_1, \ldots, i_{s-1})$
and set recursively 
$\Phi_{\beta}^{\eta} (\,\cdot\, ;G)
= \Phi_{i_s}^{\eta_s} (\,\cdot\, , \Phi_{\beta'}^{\eta} (\,\cdot\, ;G))$.

\begin{proposition}\label{pr.div_order}
There exists a constant $\nu >0$ such that
\[
\| \delta_a (X^{\ve}_1) \|_{2, -2n}= O (\ve^{-\nu})
\qquad \mbox{as $\ve\searrow 0$.}
\]
\end{proposition}

\begin{proof}
If $\{g_k\}_{k\in \N}$ is a sequence of 
smooth functions supported in $U$ such that
$g_k \to \delta_a$ in $\cS^\prime (\R^n)$ as $k\to\infty$,
then we see from the results in \cite{ta} that
\begin{align}
{\mathbb E}[ \delta_a (X^{\ve}_1) G]
&=
\lim_{k\to\infty}{\mathbb E}[ g_k (X^{\ve}_1) G]
\nn\\
&=
\lim_{k\to\infty}{\mathbb E}[ g_k (X^{\ve}_1) \eta_1 (X^{\ve}_1) G]
=
\lim_{k\to\infty}{\mathbb E}[ g_k (F_{\ve}) \eta_1 (F_{\ve}) G]
\label{eq.0621-2}
\end{align}
for every $G \in \mathbf{D}_\infty$,
where we wrote $F_{\ve}=\hat{X}^{\ve}_1$ again.

Next, recall that if we set $f(x) := \prod_{j=1}^n (x_j- a_j)^+$
and $\beta =(1,1,2,2, \ldots, n,n)$,  
then $\partial^{\beta} f = \delta_a$ in the distribution sense.
Take $\kappa >0$ so small that 
$\prod_{j=1}^n [a_j -\kappa, a_j +\kappa] \subset U$ holds.
For this $\kappa$, we can find a sequence $\{\lambda_k\}_{k=1}^\infty$
of  smooth and non-decreasing function on $\R$ such that
(1) $\lambda_k (z)$ coincides with $z^+ :=z\vee 0$ outside 
$[-\kappa, \kappa]$ for all $k$,
(2)  $\lambda_k (z)$ converges to $z^+$
uniformly on $[-\kappa, \kappa]$ as $k\to\infty$.
Set $f_k (x) = \prod_{j=1}^n \lambda_k (x_j -a_j)$.
Then $f_k$ is smooth on $\R^n$ and  $\lim_{k\to\infty} f_k =f$
uniformly on $\R^n$.
Note that $\lim_{k\to\infty}\partial^\beta f_k =\delta_a$  
in the distribution sense and $\partial^\beta f_k \equiv 0$ outside $U$.
So, we may take $g_k = \partial^\beta f_k$ in \eqref{eq.0621-2} above.

By \eqref{eq.loc.ibp} in Lemma \ref{lem.0624} 
and the way $\Phi_{\beta}^{\eta} (\,\cdot\, ;G)$ is defined, 
we have
\begin{align}%
{\mathbb E}[\partial^{\beta} f_k (F_{\ve}) \eta_1 (F_{\ve}) G]
&=
{\mathbb E}[ f_k (F_{\ve})  \eta_{2n+1} (F_{\ve})
\Phi_{\beta}^{\eta}(\,\cdot\,;G) ].
\nn
\end{align}
From this and \eqref{eq.0621-2}, we have
\begin{align}\label{loc.ibp_rpt}
{\mathbb E}[ \delta_a (X^{\ve}_1) G]
&=
{\mathbb E}[ f (F_{\ve})  \eta_{2n+1} (F_{\ve})
\Phi_{\beta}^{\eta}(\,\cdot\,;G) ].
\end{align}
Note that $|f (F_{\ve})  \eta_{2n+1} (F_{\ve}) |$ 
is dominated by  a polynomial in $|F_{\ve}|$.
So, $\|f (F_{\ve})  \eta_{2n+1} (F_{\ve}) \|_{L^p}$ is bounded in $\ve$ for every $1<p<\infty$.
By using \eqref{eq.loc.ibp} and  \eqref{eq.0625-1}
 in Lemma \ref{lem.0624}  repeatedly, 
we can easily show that
\[
\| \Phi_{\beta}^{\eta} (\,\cdot\, ;G) \|_{L^{3/2}} = c \ve^{-\nu} \|G\|_{2,2n},
\qquad
0<\ve \le 1,
\]
if the value of $\nu >0$ is adjusted.
Therefore, the right hand side of \eqref{loc.ibp_rpt} is also 
dominated by $c \ve^{-\nu} \|G\|_{2,2n}$, 
where $c$ is independent of $\ve$ and $k$.  
This proves the proposition.
\end{proof}

\begin{proposition}\label{pr.exp_tight}
The family $\{ \nu^{\ve}_{u,a}\}_{0<\ve \le 1}$
is exponentially tight on $G\Omega^B_{\al, 4m} ({\mathbb R}^d)$, 
that is,
for every $M \in (0,\infty)$, there exists a compact set $K=K_M$ 
such that 
\[
\limsup_{\ve \searrow 0 } \ve^2 \log \nu^{\ve}_{u,a} (K^c)
\le -M.
\]
\end{proposition}

\begin{proof}
Let $\hat{B}_{R}= \hat{B}_{R}^{\alpha, 4m}$ be ``ball of radius $R>0$" as in \eqref{def.hatball}.
It is shown in  \cite[Lemma 7.6]{in1} that 
there exists a constant $c=c_{\alpha, 4m} >0$ independent of $R$ such that
\[
{\rm Cap}_{2,2n} (\{w \in W \mid \cL (w) \in  ( \hat{B}_{R}^{\alpha, 4m})^c\}) \le e^{-c R^2}
\qquad
\mbox{for sufficiently large $R >0$.}
\]

Take $\kappa >0$ so small  that Besov parameter
$(\alpha +\kappa, 4m)$ still satisfies \eqref{eq.amam}.
It is well-known that  
$G\Omega^B_{\alpha +\kappa, 4m} ( {\mathbb R}^d)$
 is embedded in $G\Omega^B_{\alpha, 4m} ( {\mathbb R}^d)$
and  every bounded subset of the former is 
precompact in the latter.

By the way $\nu^{\ve}_{u,a}$ is defined, it holds that 
\begin{align*} 
\nu^{\ve}_{u,a} ( ( \hat{B}_{R}^{\alpha+\kappa, 4m})^c)
&=
\theta^{\ve}_{u,a} (\{ w \in \cW \mid  \ve \cL (w) \in 
( \hat{B}_{R}^{\alpha+\kappa, 4m})^c \})
\nn\\
&\le 
\| \delta_a (X^{\ve}_1) \|_{2, -2n} 
{\rm Cap}_{2,2n} (\{w \in W\mid \cL (w) \in  ( \hat{B}_{R/\ve}^{\alpha+\kappa, 4m})^c\}) 
\nn\\
&\le C \ve^{-\nu}e^{-c R^2/\ve^2}
\qquad
\mbox{for sufficiently large $R >0$.}
 \end{align*} 
Here, we used the inequality in Item {\bf (e)} in Subsection \ref{subsec.MC} and
Proposition \ref{pr.div_order}.
Hence, we have
\[
\limsup_{\ve \searrow 0 } \ve^2 \log \nu^{\ve}_{u,a} 
(( \hat{B}_{R}^{\alpha+\kappa, 4m})^c)
\le -c R^2.
\]
For given $M>0$, we choose $R$ 
so large that $cR^2 \ge M$ holds.
Since $ \hat{B}_{R}^{\alpha+\kappa, 4m}$ is precompact 
in $(\alpha, 4m)$-Besov topology, the proof is completed.
\end{proof}

For $R>0$ and ${\bf w} \in G\Omega^B_{\al, 4m} ({\mathbb R}^{d})$, we set 
\begin{equation}\label{def.hatball2}
\bar{B}_{R} ({\bf w})=\{ {\bf v} \in G\Omega^B_{  \al, 4m} ({\mathbb R}^{d}) 
\mid
\|{\bf v}^1 -{\bf w}^1\|_{ \al, 4m -B}^{4m} 
 +  \|{\bf v}^2 - {\bf w}^2 \|_{2 \al, 2m -B}^{2m} \le R^{4m}
 \}.
\end{equation}
Clearly, 
$\{ \bar{B}_{R} ({\bf w}) |R>0\}$ forms a fundamental system of  neighborhoods around ${\bf w}$.
Set 
\[
\Xi_\ve = \|(\ve {\bf W})^1 -{\bf w}^1\|_{ \al, 4m -B}^{4m} 
 +  \|(\ve{\bf W})^2 - {\bf w}^2 \|_{2 \al, 2m -B}^{2m}.
 \]
 Then, for every $\ve \in (0,1]$ and ${\bf w}$,
 $\Xi_\ve = \Xi_\ve (\,\cdot\, ,{\bf w} )$ is $\infty$-quasi continuous 
 and belongs to $\mathbf{D}_\infty$.
Moreover, $\{ \Xi_\ve\}_\ve$ is bounded in $\ve$ in $\mathbf{D}_{p,k}$
for every $1<p<\infty$ and $k \in \N$.
For a smooth,
non-increasing function $\chi \colon [0,\infty) \to [0,1]$ 
such that $\chi \equiv 1$ on $[0,1]$
and $\chi \equiv 0$ on $[2, \infty)$, 
$\chi (\Xi_\ve/ R^{4m}) \in \mathbf{D}_\infty$ satisfies
${\bf 1}_{\bar{B}_{R} ({\bf w})} \circ (\ve\cL) \le \chi (\Xi_\ve/ R^{4m}) 
  \le {\bf 1}_{\bar{B}_{2R} ({\bf w})} \circ (\ve\cL)$ quasi-surely.

Let us recall that  the law of $\ve {\bf W} = \ve \cL (w)$ satisfies 
the standard version 
of Schilder-type LDP on $G\Omega^B_{\al, 4m} ({\mathbb R}^{d})$
with a good rate function $I$, where
\begin{align}
I ({\bf w}) 
= 
\begin{cases}
  \tfrac12  \|h\|^2_{\cH} & (\mbox{if ${\bf w}= {\cal L}(h)$ for some 
  $h \in \cH$}), \\
    \infty &  (\mbox{otherwise}).
  \end{cases}
\nn
\end{align}
By a general result on LDPs (see \cite[Lemma 4.1.6]{dzbk} for example), we then have
\begin{align}\label{eq.0625-2}
\lefteqn{
\lim_{R\searrow 0} \limsup_{\ve\searrow 0} 
 \ve^2 \log \mu  (\{w\in \cW \mid \ve {\bf W}\in \bar{B}_{R} ({\bf w})\})
} \nn\\
 &\le -  \lim_{R\searrow 0}\,
   \inf  \{  I ({\bf v}) \mid {\bf v}\in \bar{B}_{R} ({\bf w}) \}
   = - I ({\bf w}),
 \quad
 {\bf w} \in G\Omega^B_{\al, 4m} ({\mathbb R}^{d}).
 \end{align}

\begin{lemma}\label{lem.0630}
Let the notation be as above. For every $p \in (1,\infty)$ 
and $k \in \N$, it holds that
\[
\lim_{R\searrow 0}
\limsup_{\ve \searrow 0}
 \ve^2 \log \| \chi (\Xi_\ve (\,\cdot\, ,{\bf w} )/ R^{4m}) \|_{p,k}^p
 \le 
 - I ({\bf w}),
  \qquad
 {\bf w} \in G\Omega^B_{\al, 4m} ({\mathbb R}^{d}). 
 \]
\end{lemma}

\begin{proof}
We will write $\Xi_\ve = \Xi_\ve (\,\cdot\, ,{\bf w} )$.
By Meyer's equivalence, it is enough to estimate 
$\sum_{j=0}^k \|D^j \chi (\Xi_\ve/ R^{4m})  \|_{L^p}^p$,
where $D$ is the $\cH$-derivative. Hence, it amounts to compute 
\[
\max_{0\le j \le k} \limsup_{\ve \searrow 0}
 \ve^2 \log \|D^j \chi (\Xi_\ve/ R^{4m})  \|_{L^p}^p.
\]

For $j=0$, we can easily see that
$\|\chi (\Xi_\ve/ R^{4m})  \|_{L^p}^p \le
 \mu (\{\ve{\bf W} \in \bar{B}_{2R} ({\bf w} )\})$.
 For $j=1$, we have $D\chi (\Xi_\ve/ R^{4m}) 
 = \chi^\prime (\Xi_\ve/ R^{4m}) (D\Xi_\ve)  R^{-4m}$.
Hence, for every $q \in (1,\infty)$ and $R>0$, 
there exists a positive 
constant $C=C_{q, R}$ (independent of $\ve$) such that
\begin{align*}
\|D \chi (\Xi_\ve/ R^{4m})  \|_{L^p}^p 
&\le
{\mathbb E}[ |\chi^\prime (\Xi_\ve/ R^{4m}) |^{pq}]^{1/q} 
{\mathbb E}[ \|D\Xi_\ve\|_{\cH}^r  R^{-4mr} ]^{1/r} 
\nn\\
&\le
C\mu (\{\ve{\bf W} \in \bar{B}_{2R} ({\bf w}) \})^{1/q},
\end{align*}
where $1/q +1/r =1$.
By repeating similar computations, we have for all $0 \le j \le k$ that
 \[
\|D^j \chi (\Xi_\ve/ R^{4m})  \|_{L^p}^p 
\le C\mu (\{\ve{\bf W} \in \bar{B}_{2R} ({\bf w} \})^{1/q}.
\]
Using \eqref{eq.0625-2}, we have
$
\lim_{R\searrow 0}
\limsup_{\ve \searrow 0}
 \ve^2 \log \| \chi (\Xi_\ve/ R^{4m}) \|_{p,k}^p
 \le 
 - (1/q) I ({\bf w})$.
 Letting $q \searrow 1$, we finish the proof.
\end{proof}

\begin{proposition}\label{pr.limsup_ball}
Let the notation be as above. Then, we have 
\begin{align}
\lim_{R\searrow 0} \limsup_{\ve\searrow 0} 
 \ve^2 \log \nu^{\ve}_{u,a}  (\bar{B}_{R} ({\bf w}))
  \le - I_1 ({\bf w}),
 \qquad
 {\bf w} \in G\Omega^B_{\al, 4m} ({\mathbb R}^{d}).
 \label{eq.0630-1}
 \end{align}
\end{proposition}

\begin{proof}
Write $\hat{a} :=\Psi ({\bf w}, 0)_1$ for simplicity.
First, we consider the case $a \neq \hat{a}$.
Let  $U$ and $\hat{U}$ (with $U\cap \hat{U}=\emptyset$) be 
a neighborhood of $a$ and $\hat{a}$, respectively.
By Lyons' continuity theorem,
there exists $R_0 >0$ and $\ve_0 >0$
such that 
$\Psi ({\bf v}, \lambda^\ve)_1 \in  \hat{U}$ 
if ${\bf v} \in \bar{B}_{2R_0} ({\bf w})$ and $\ve \in (0,\ve_0)$.
Suppose that $\{ g_k\}_{k=1}^\infty$ be a sequence of smooth functions
with supports in $U$ that approximates $\delta_a$. 
Then,  we have
\begin{align*}
\nu^{\ve}_{u,a}  (\bar{B}_{R} ({\bf w}))
&\le
{\mathbb E} [\chi (\Xi_\ve/ R^{4m}) \delta_a (X^{\ve}_1)]
=
\lim_{k\to\infty} {\mathbb E} 
[\chi (\Xi_\ve/ R^{4m}) g_k (X^{\ve}_1)]
\\
&\le
\lim_{k\to\infty} {\mathbb E}
 [{\bf	 1}_{\{\ve{\bf W} \in \bar{B}_{2R} ({\bf w}) \}} \,
    g_k (\Psi (\ve{\bf W}, \lambda^\ve)_1)] =0
 \end{align*}
if $\ve \in (0,\ve_0)$ and $R \in (0, R_0)$.
Thus, we  have shown \eqref{eq.0630-1} for this case.

Next, we consider the case $a= \hat{a}$.
As in the proof of Proposition \ref{pr.div_order}, 
we use the localized IbP formula.
Let $\eta := \{\eta_i\}_{i=1}^{2n+1}$ as in Proposition \ref{pr.div_order}.
Then, we can see that, for every $p \in (1,\infty)$, 
there exist positive constants $c$ and $\nu$ 
(independent of $\ve$) such that
\begin{align*}
\nu^{\ve}_{u,a}  (\bar{B}_{R} ({\bf w}))
&\le
{\mathbb E} [\chi (\Xi_\ve/ R^{4m}) \delta_a (X^{\ve}_1)]
\nn\\
&=
{\mathbb E}[ f (F_{\ve})  \eta_{2n+1} (F_{\ve})
\Phi_{\beta}^{\eta}(\,\cdot\,; \chi (\Xi_\ve/ R^{4m})) ]
\nn\\
&\le c \ve^{-\nu} \| \chi (\Xi_\ve/ R^{4m}) \|_{p,2n}.
 \end{align*}
Here, we used \eqref{eq.0625-1} and 
\eqref{eq.loc.ibp} in Lemma \ref{lem.0624} repeatedly.
Using Lemma \ref{lem.0630} and then letting $p \searrow 1$,
we have 
$
\lim_{R\searrow 0} \limsup_{\ve\searrow 0} 
 \ve^2 \log \nu^{\ve}_{u,a}  (\bar{B}_{R} ({\bf w}))
  \le - I ({\bf w})
  $ 
  in this case.
Thus,  we have proved  \eqref{eq.0630-1}.
\end{proof}

\appendix\section{Positivity of heat kernel}
\label{sec.appen}

Let $p^{\ve}_t (x, a)  = 
{\mathbb E}[\delta_a (X^{\ve}_t ) ]$ be the heat kernel 
(or the density function) on $\cM$
associated with the 
$\ve^2 (\Delta_{\mathrm{sub}}/2 +V)$-diffusion process.
Here, $(X^{\ve}_t)$ is the projection of the
solution of SDE \eqref{eq.sde_on_P} with the starting point $u$.
The skeleton ODE with drift 
which corresponds to SDE \eqref{eq.sde_on_P} 
with $\ve =1$ is given as follows:
\begin{equation}\label{ode.V0}
d\bar\phi(h)_t 
= \sum_{i=1}^d  A_i ( \bar\phi(h)_t) dh^i_t + A_0 ( \bar\phi(h)_t)dt, 
\qquad   \bar\phi(h)_0 =u.
\end{equation}
We write $\bar\psi(h)_t=\pi( \bar\phi(h)_t)$.

The purpose of this appendix is to verify that
\begin{equation}\label{eq.pos}
p^{\ve}_t (x, a) >0 \qquad \mbox{for every
$x, a \in \cM$ and $\ve, t \in (0,1]$.}
\end{equation}
By the scaling property $p^{\ve}_t (x, a) =p^{1}_{\ve^2 t} (x, a)$, 
 it is enough to prove  \eqref{eq.pos} for $\ve =1$.
This kind of positivity for the density
of Wiener functionals has been well-studied (see \cite{aks}
and Remark \ref{rem_aks} below).
According to these result, Lemma \ref{lem.0701} below implies \eqref{eq.pos}.

\begin{lemma}\label{lem.0701}
Let $x, a \in \cM$ and $u \in \pi^{-1} (x)$.
Then, for every $\tau \in (0,1]$, there exists $h\in \cH$ such that 
$ \bar\psi(h)_\tau =a$ and $D\bar\psi (h)_\tau \colon \cH \to T_a\cM$ 
is surjective.
\end{lemma}

\begin{proof}
It is sufficient to prove the case $\tau =1$. The general case can be done with trivial modifications.

We now make two simple remarks.
First, let $\{Z_1,\ldots, Z_d \}$ be an orthonormal frame of $\cD$
on a coordinate neighborhood $U\subset \cM$.
Consider the following ODE on $U$ driven by a Cameron-Martin path $h$:
\begin{equation}\label{ode.zeta2}
d\zeta(h)_t 
= \sum_{i=1}^d  Z_i ( \zeta (h)_t) dh^i_t + V_0 ( \zeta (h)_t)dt.
\end{equation}
Then, $\zeta =\zeta (h)$ is of finite energy and  satisfies that
$\zeta^\prime_t - V_0 (\zeta_t) \in \cD_{\zeta_t} $ for almost all $t$.
Conversely, if $\zeta$ is a path on $\cM$ of finite energy
such that $\zeta^\prime_t - V_0(\zeta_t) \in \cD_{\zeta_t} $ for almost all $t$, then there uniquely exists a Cameron-Martin path $h$
such that $\zeta =\zeta (h)$. In this sense, we have a one-to-one
correspondence $h\leftrightarrow\zeta(h)$.

Second, consider $\zeta \in \cH_x (\cM, T\cM)$ such that 
$\zeta^\prime_t - V_0(\zeta_t) \in \cD_{\zeta_t} $ for almost all $t \in [0,1]$.
We denote by $\xi$ the horizontal lift of $\zeta$ with $\xi_0=u$.
Since $A_0$ is the horizontal lift of $V_0$, 
$\xi^\prime_t - A_0(\xi_t) \in \cK_{\xi_t}$ for almost all $t \in [0,1]$.
Since $\{ A_1(u), \ldots, A_d (u)\}$ forms a linear basis 
of $\cK_u$ for all $u \in \cP$, 
there uniquely exists $h \in \cH$ such that 
$\xi =\bar\phi(h)$.
Conversely, if $\bar\phi(h)$ solves 
ODE \eqref{ode.V0}, then $\zeta := \bar\psi(h)$ satisfies that
$\zeta^\prime_t - V_0(\zeta_t) \in \cD_{\zeta_t} $ for almost all $t$.
In this way,  we have a one-to-one
correspondence $h\leftrightarrow\zeta $.

By using 
Proposition \ref{pr.kunita} and Lemma \ref{lem.dMC_dense},
 we can now prove Lemma \ref{lem.0701} 
in a similar way to the proof of Lemma \ref{pr.CMnice}.

Let $U$ is a coordinate neighborhood of $x$ 
and consider ODE \eqref{ode.zeta2} with $\zeta (h)_0 =x$ on $U$.
By extending the coefficient vector fields with compact support, 
we also view \eqref{ode.zeta2} as an ODE on $\R^n$.
Since these vector fields satisfy 
H\"ormander's condition at $x$, we can use  Lemma \ref{lem.dMC_dense}. It implies that
 there exists a Cameron-Martin
path $k \colon [0,1/2] \to \R^d$ such that (1) the tangent map of
$\zeta (\,\cdot\,)_{1/2}$ is surjective at $k$
and (2) $[0,1/2]\ni t\mapsto \zeta (k)_t$ stays inside $U$.
By Proposition \ref{pr.kunita}, there exist a finite energy path
$\eta\colon [1/2,1] \to \cM$ such that 
$\eta_{1/2} = \zeta (k)_{1/2}$, $\eta_1 =a$ and 
$\eta^\prime_t - V_0(\eta_t) \in \cD_{\eta_t} $ for almost all $t$.

Define $\zeta \in \cH_x (\cM, T\cM)$ to be the concatenation 
of $\zeta (k)$ and $\eta$.
Then, $\zeta_1 =a$ and $\zeta^\prime_t - V_0(\zeta_t) \in \cD_{\zeta_t} $ for almost all $t$.
The corresponding $h \in \cH$ is the desired element.
(The proof of the surjectivity of $D\bar\psi (h)_1$ is essentially 
the same as in the proof of Lemma \ref{pr.CMnice}
and is therefore omitted.)
\end{proof}

\begin{remark}\label{rem_aks}
Precisely, the positivity theorem for the density 
in \cite{aks} is for SDEs on a Euclidean space.
But, after slightly modifying it, one can verify that 
it still holds  for SDEs on a compact manifold.
\end{remark}

\noindent
{\bf Acknowledgement:}~
The author is partially supported by 
JSPS KAKENHI Grant No. 20H01807.


\bigskip
\begin{flushleft}
  \begin{tabular}{ll}
    Yuzuru \textsc{Inahama}
    \\
    Faculty of Mathematics,
    \\
    Kyushu University,
    \\
    744 Motooka, Nishi-ku, Fukuoka, 819-0395, JAPAN.
    \\
    Email: {\tt inahama@math.kyushu-u.ac.jp}
  \end{tabular}
\end{flushleft}

\end{document}